\documentclass[reqno,12pt]{amsart}
\usepackage{amscd}
\usepackage{amssymb}
\usepackage{amsmath}
\usepackage{amsfonts}
\usepackage{enumerate}
\usepackage{graphicx}
\usepackage{epic}
\usepackage[all,cmtip]{xy}
\usepackage{overpic}
\usepackage{subfig}
\usepackage{caption}
\usepackage{tikz}
\usepackage[vcentermath]{youngtab}

\newtheorem{thm}{Theorem}[section]
\newtheorem{prop}[thm]{Proposition}
\newtheorem{lem}[thm]{Lemma}
\newtheorem{cor}[thm]{Corollary}

\theoremstyle{definition}

\newtheorem{rem}[thm]{Remark}
\newtheorem*{rem*}{Remark}

\newtheorem*{nota}{Notations}

\newcommand{\mm}{\mathbf{m}}
\newcommand{\nn}{\mathbf{n}}

\newcommand{\yy}{\mathbf{y}}

\newcommand{\DD}{\mathcal {D}}

\newcommand{\Zz}{\mathbb {Z}}
\newcommand{\Cc}{\mathbb {C}}

\newcommand{\im}{\mathrm{Im\,}}
\newcommand{\coker}{\mathrm{coker\,}}

\def\w{\widetilde}

\def\xr{\xrightarrow}

\numberwithin{equation}{section}
\usepackage[colorlinks,linktocpage]{hyperref}
\hypersetup{citecolor=blue,linkcolor=blue}
\begin{document}
	\title[The cohomology of $BPU(4)$]{The cohomology of the classifying space of $PU(4)$}
	\author[F.~Fan]{Feifei Fan}
	\thanks{The author is supported by the National Natural Science Foundation of China (Grant no. 12271183) and by the GuangDong Basic and Applied Basic Research Foundation (Grant no. 2023A1515012217).}
	\address{Feifei Fan, School of Mathematical Sciences, South China Normal University, Guangzhou, 510631, China.}
	\email{fanfeifei@mail.nankai.edu.cn}
	\maketitle
	\begin{abstract}
Let $BPU(n)$ be the classifying space of the projective unitary group $PU(n)$. We determine the integral cohomology ring of $BPU(4)$, and the Steenrod algebra structure of its mod $2$ cohomology.
	\end{abstract}

	\section{Introduction}\label{sec:introduction}
	Let $G$ be a topological group. 
	The classifying space $BG$ of  $G$ is characterized as
	the base space of a universal $G$-bundle $G\to EG\to BG$.
	The cohomology of $BG$ is the main tool in the theory of classifying principal $G$-bundles.
	
	Let $U(n)$ be the group of $n\times n$ unitary matrices. The projective unitary group $PU(n)$ is the quotient group $U(n)/S^1$ of $U(n)$ by its center subgroup $S^1=\{e^{i\theta}I_n:\theta\in [0, 2\pi]\}$, where $I_n$
	denotes the identity matrix. $PU(n)$ is homotopy equivalent to $PGL(n,\Cc)$, the projective general linear group $GL(n,\Cc)/\Cc^*$, which is the
	automorphism group of the algebra of $n\times n$ complex matrices. Therefore, for a pointed space $X$, the set of pointed homotopy classes of
	maps $[X,BPU(n)]$ also
	classifies bundles of $n\times n$ complex matrix algebras over $X$, known as the topological Azumaya algebras of degree $n$ over $X$. (See \cite{Gro68}.)
	
	The cohomology ring of $BPU(n)$ also plays a significant role in the study of the topological period-index
	problem, which was introduced by Antieau and Williams \cite{AW14a,AW14b} as an analogue of period-index problems in
	algebraic geometry (cf. \cite{Art82,Jon04,Lieb08}). Gu \cite{Gu19,Gu20}, Crowley and Grant \cite{CG20} also investigated this problem for certain topological spaces.
	
	Since $PU(n)$ can also be viewed as the quotient group of the special unitary group $SU(n)$ by its center subgroup $\Zz/n$ generated by $e^{2\pi i/n}I_n$, there is an induced fibration of classifying spaces:
	\[B(\Zz/n)\to BSU(n)\to BPU(n).\]
	Hence, for a commutative ring $R$, if $1/n\in R$, then 
	\[H^*(BPU(n);R)\cong H^*(BSU(n);R)\cong R[c_2,c_3\dots,c_n],\ \ \deg(c_i)=2i.\] 
	It follows that if $x\in H^*(BPU(n);\Zz)$ is a torsion element, then there exists
	$k\geq 0$ such that $n^kx=0$. (In the case of Chow rings, Vezzosi \cite{Vezz00} proved the stronger result that
	all torsion classes in the Chow ring of $BPGL(n,\Cc)$ are $n$-torsion.)
	
	However, if $n$ is not invertible in the coefficient ring $R$, the ring structure of $H^*(BPU(n);R)$ is very complicated for general $n$.
	Here we list some of the known results on the cohomology of $BPU(n)$ for special values of $n$. 
	
	Since $PU(2)=SO(3)$, the ring $H^*(BPU(2);\Zz)$ is well understood as a special case of the result of Brown \cite{Brow82}.
	
	The ring $H^*(BPU(3);\Zz/3)$ is computed in Kono-Mimura-Shimada \cite{KMS75}. The Steenrod algebra structure of $H^*(BPU(3);\Zz/3)$ and the Brown-Peterson cohomology of $BPU(3)$ are determined in Kono-Yagita \cite{KY93}.
	
	The ring $H^*(BPU(n);\Zz/2)$ is known  if $n\equiv 2$ mod $4$ (Kono-Mimura \cite{KM75} and Toda \cite{Tod87}). In \cite{Tod87}, Toda also computed  the ring $H^*(BPU(4);\Zz/2)$.
	
	In \cite{Vis07}, Vistoli provided a description of  $H^*(BPU(p);\Zz)$, as well as the Chow ring of $BPGL(p,\Cc)$, for any odd prime $p$. In particular, the ring structure of $H^*(BPU(3);\Zz)$ is given in terms of generators and relations there. We will explicitly explain Vistoli's result in the next Section. 
	
	More interesting results on the ordinary cohomology ring $H^*(BPU(p);\Zz/p)$ (resp. on the Brown-Peterson cohomology of $BPU(p)$) for an arbitrary prime $p$ can be found in Vavpeti\v{c}-Viruel \cite{VV05} (resp. in Masaki-Yagita \cite{KY08}).

	For an arbitrary integer $n>3$, the cohomology of $BPU(n)$ is only known in certain finite range of dimensions. For instance, Gu \cite{Gu21} determined the ring structure of $H^*(BPU(n);\Zz)$ in dimensions less
	than or equal to $10$. Other partial results on the $p$-local cohomology groups of $BPU(n)$ for an odd prime $p$, $p\mid n$,
	can be found in Gu-Zhang-Zhang-Zhong \cite{GZZZ22} and Zhang-Zhang-Zhong \cite{ZZZ23}.
	
	Summarizing these results, the only examples of $BPU(n)$, whose integral cohomology rings can be described via generators and relations, are the cases $n=2,3$. Also, these are the only cases for which the Steenrod algebra structure of $H^*(BPU(n);\Zz/p)$, $p$ a prime such that $p\mid n$, are known. The main goal of this paper is to determine the integral cohomology ring of $BPU(4)$ and the Steenrod algebra structure of its mod $2$ cohomology.
	
\section{Main results}\label{sec:results}
\begin{nota}
	For an integer $n>0$, let $\Lambda_n$ be the ring of symmetric polynomials over $\Zz$ in $n$ variables. It is known that $\Lambda_n=\Zz[\sigma_1,\dots,\sigma_n]$, where $\sigma_i$ is the $i$th elementary symmetric polynomial in $n$ variables. Let $\Lambda_n^d\subset \Lambda_n$ be the $\Zz$-submodule of symmetric polynomials of degree $d$.
	
	Let $\nabla=\sum_{i=1}^n\partial/\partial v_i$ be the linear differential operator acting on the polynomial algebra $\Zz[v_1,\dots,v_n]$. 
	It is easy to see that $\nabla$ on $\Zz[v_1,\dots,v_n]$ preserves symmetric polynomials, so $\nabla$ can be restricted to $\Lambda_n$. 
	Let $K_n$ be the kernel of $\nabla$ on $\Lambda_n$. Generally, for a commutative ring $R$ with unit, we use the notation $\nabla_R$ to denote the map $\nabla\otimes R$. 
\end{nota}	

For the  graded polynomial ring $\Zz[v_1,\dots,v_n]$ with $\deg(v_i)=2$, define a graded algebraic homomorphism $\Theta_n:\Zz[v_1,\dots,v_n]\to\Zz[\eta]/(n\eta)$, $\deg(\eta)=2$, by $\Theta(v_i)=i\eta$. In fact, this is the ring homomorphism $H^*(BT^n;\Zz)\to H^*(B(\Zz/n);\Zz)$ induced by the embedding $\Zz/n\hookrightarrow T^n$, $\omega\mapsto (\omega,\omega^2,\dots,\omega^{n-1},1)$, $\omega=e^{2\pi i/n}$, $T^n=(S^1)^n$.
Let $\Theta'_n$ be the restriction of $\Theta_n$ to $K_n$, and let $K'_n=\ker\Theta'_n$. The result of Vistoli \cite{Vis07} can be described as follows.

\begin{thm}[Vistoli {\cite[Theorem 3.4]{Vis07}}]\label{thm:vistoli}
	For any prime $p>2$, the integral cohomology ring of $BPU(p)$ is given by
	\[ H^*(BPU(p);\Zz)\cong\frac{K_p\otimes \Zz[x_3,x_{2p+2}]}{(x_3^2,\,px_3,\,px_{2p+2},\,x_3K_p',\,x_{2p+2}K'_p)},\ \deg(x_i)=i.\]
\end{thm}

For the element $\delta=\prod_{i\neq j}(v_i-v_j)\in K_p$, one easily checks that 
\[\Theta'_p(\delta)\equiv -\eta^{p^2-p}\neq 0\mod (p\eta).\]
It is shown in \cite[Proposition 3.1]{Vis07} that the image of $\Theta'_p$ is generated by $\Theta'_p(\delta)$, so for $p>2$, as $\Zz$-module, 
\[H^*(BPU(p);\Zz)\cong K_p\oplus \big((E_{\Zz/p}[x_3]\otimes\Zz/p[x_{2p+2}])^+\otimes\Zz[\delta]\big),\]
where $E_{\Zz/p}[x_3]$ means the exterior algebra over $\Zz/p$ with the generator $x_3$. Since the rank of $K_p$ is the same as the rank of $H^*(BSU(n);\Zz)\cong\Zz[c_2,\dots,c_n]$, the additive structure of $H^*(BPU(p);\Zz)$ are determined completely.

One reason for the difficulty of describing $H^*(BPU(n);\Zz)$  is that the ring  $K_n$ is hard to compute for large $n$. For small values of $n$, here are some results.
$K_2\cong \Zz[\alpha_2]$, $\alpha_2=\sigma_1^2-4\sigma_2\in\Lambda_2^2$, is obvious. Vezzosi \cite[Lemma 3.2]{Vezz00} (see also \cite[Theorem 14.2]{Vis07}) proved that \[K_3\cong \Zz[\alpha_2,\alpha_3,\alpha_{6}]/(27\alpha_{6}-4\alpha_2^3-\alpha_3^2),\ \ \alpha_i\in \Lambda_3^i.\]
Here we give the description of $K_4$.
\begin{thm}\label{thm:K_4}
There is a ring isomorphism
\[K_4\cong\Zz[\alpha_2,\alpha_3,\alpha_{4},\alpha_{6}]/(64\alpha_6-97\alpha_2^3-27\alpha_3^2+144\alpha_2\alpha_4),\ \ \alpha_i\in\Lambda_4^i.\] 
\end{thm}
For $n\geq 5$, the complexity of $K_n$ grows very quickly as $n$ grows. To the author's knowledge, there is no one computed $K_n$ for $n\geq 5$ by writing down the generators and relations explicitly so far.

As in the case of $K_p$ for $p$ a prime, $K_4$ can be viewed as the quotient ring of $H^*(BPU(4);\Zz)$ by torsion elements.
\begin{thm}\label{thm:BPU4}
The cohomology ring $H^*(BPU(4);\Zz)$ is isomorphic to the graded ring
\[K_4\otimes\Zz[x_3,x_{10},x_{15}]/I,\ \ \deg(\alpha_i)=2i,\ \deg(x_i)=i,\]
where $I$ is the ideal generated by
\begin{gather*}
	4x_3,\ 2x_3^2,\ \alpha_2x_3,\ \alpha_3x_3,\ 
	2x_{10},\ \alpha_2x_{10},\ \alpha_3x_{10},\ 2x_{15},\ \alpha_2x_{15},\ \alpha_3x_{15},\\
	x_3^6\alpha_6+x_3^4x_{10}\alpha_4+x_3^{5}x_{15}+x_{10}^3+x_{15}^2.	
\end{gather*}
\end{thm}

One of the results in Toda \cite{Tod87} gives  $H^*(BPSO(4m+2);\Zz/2)$, $m\geq 1$.  In particular, this gives the mod $2$ cohomology of $BPU(4)$, since it is known that $PU(4)$ is homeomorphic to $PSO(6)$. 
\begin{thm}[Toda {\cite[(4.10)]{Tod87}}]\label{thm:toda}
	The mod $2$ cohomology of $BPU(4)$ is given by
	\[H^*(BPU(4);\Zz/2)\cong\Zz/2[y_2,y_3,y_5,y_8,y_9,y_{12}]/J,\ \ \deg(y_i)=i,\]
where $J$ is the ideal generated by 
	\[y_2y_3,\ y_2y_5,\ y_2y_9,\ y_9^2+y_3^2y_{12}+y_5^2y_8.\]
\end{thm}
In the notation of Theorem \ref{thm:toda}, we determine the Steenrod algebra structure of $H^*(BPU(4);\Zz/2)$ as follows.
\begin{thm}\label{thm:steenrod}
	$Sq^1(y_2)=Sq^1(y_3)=0$, $Sq^1(y_5)=y_3^2$, $Sq^1(y_8)=y_3^3$, $Sq^1(y_9)=y_5^2$, $Sq^1(y_{12})=y_3y_5^2$; 
	$Sq^2(y_3)=y_5$, $Sq^2(y_5)=0$, $Sq^2(y_8)=y_5^2$, $Sq^2(y_9)=y_3^2y_5$, $Sq^2(y_{12})=y_2y_{12}+y_3^2y_8$; 
	$Sq^4(y_5)=y_3^3+y_9$,  $Sq^4(y_8)=y_2^2y_8+y_{12}+y_3^4$, $Sq^4(y_9)=y_3y_5^2$, $Sq^4(y_{12})=y_2^2y_{12}+y_3^2y_5^2$; 
	$Sq^8(y_9)=y_3y_5y_9+y_5y_{12}+y_8y_9$, $Sq^8(y_{12})=y_3^4y_8+y_8y_{12}$.
\end{thm}

 In the next two sections, we review Gu's work in \cite{Gu21}, where he constructs a fibration  with $BPU(n)$ as the total space, and provides an approach to calculate the Serre spectral sequence associated to this fibration. 

	\section{Serre spectral sequences in the computation of $H^*(BPU(n))$}\label{sec:SS}
	
	\subsection{The spectral sequence $^UE$}
	The short exact sequence of Lie groups
	\[1\to S^1\to U(n)\to PU(n)\to 1\]
	induces  a fibration of their classifying spaces 
	\begin{equation}\label{eq:map pi}
		BS^1\to BU(n)\xrightarrow{\pi} BPU(n).
	\end{equation}
	Since $BS^1$ is of the homotopy type of the Eilenberg-Mac Lane space $K(\Zz,2)$, we obtain another induced fibration:
	\begin{equation}\label{eq:map chi}
		U:BU(n)\xrightarrow{\pi} BPU(n)\xrightarrow{\chi} K(\Zz,3).
	\end{equation}
	(Cf. \cite{Gan67}.) Let $^UE$ be the cohomological Serre spectral sequence associated to this fibration. The $E_2$ page of this spectral sequence has the form
	\[^UE_2^{s,t}=H^s(K(\Zz,3);H^t(BU(n)))\Longrightarrow H^{s+t}(BPU(n)).\]
	Notice that $^UE_2^{*,*}$ is concentrated in even rows,  so all $d_2$ differentials are trivial, and we have $^UE_2^{*,*}={^UE}^{*,*}_3$.
	
	\subsection{The spectral sequences $^KE$ and $^TE$}
	Let $T^n$ and $PT^n=T^n/S^1\cong T^{n-1}$ be the maximal tori of $U(n)$ and $PU(n)$ respectively, and let $\psi:T^n\to U(n)$ and $\psi':PT^n\to PU(n)$ be the inclusions.
	Then there is an exact sequence of Lie groups 
\[1\to S^1\to T^n\to PT^n\to 1,\]
which as above induces a fibration 
\[T:BT^n\to BPT^n\to K(\Zz,3).\]
We also consider the path fibration 
\[K: BS^1\simeq K(\Zz,2)\to *\to K(\Zz,3),\]
where $*$ denotes a contractible space. Let $^TE$ and $^KE$ denote the cohomological Serre spectral sequences associated  to $T$ and $K$ respectively.

Let $\varphi:S^1\to T^n$ be the diagonal map. Then there is the following homotopy commutative diagram of fibrations
\begin{gather}\label{diag:SS}
\begin{gathered}
	\xymatrix{K:\ar[d]^{\Phi}& BS^1\ar[r] \ar[d]^{B\varphi}&\ast\ar[r] \ar[d]&K(\Zz,3)\ar[d]^{=}\\
		T:\ar[d]^{\Psi} & BT^n\ar[r] \ar[d]^{B\psi}&BPT^n\ar[r] \ar[d]^{B\psi'}&K(\Zz,3)\ar[d]^{=}\\
		U: & BU(n)\ar[r]&BPU(n)\ar[r]&K(\Zz,3)}
\end{gathered}
\end{gather}

We write $^Ud_*^{*,*}$, $^Td_*^{*,*}$, $^Kd_*^{*,*}$ for the differentials of $^UE$, $^TE$, and $^KE$, respectively. We also use the simplified notation $d^{*,*}_*$ if there is no ambiguity. 

\subsection{Spectral sequence maps on $E_2$ pages}\label{subsec:maps}
Recall that the cohomology rings of the fibers $BS^1$, $BT^n$ and $BU(n)$ in \eqref{diag:SS} are
\begin{align*}
	H^*(BS^1;\Zz)&=\Zz[v],\ \deg(v)=2,\\
	H^*(BT^n;\Zz)&=\Zz[v_1,\dots,v_n],\ \deg(v_i)=2,\\
	H^*(BU(n);\Zz)&=\Zz[c_1,\dots,c_n],\ \deg(c_i)=2i.\label{eq:BPUn}
\end{align*}
Here $c_i$ is the $i$th universal Chern class of the classifying space $BU(n)$ for $n$-dimensional complex bundles.
The induced homomorphisms between these cohomology rings are given by
\[
	B\varphi^*(v_i)=v,\quad B\psi^*(c_i)=\sigma_i(v_1,\dots,v_n),
\]
where $\sigma_i(v_1,\dots,v_n)$ is the $i$th elementary symmetric polynomial in the variables $v_1,\dots,v_n$.
Hence, from \eqref{diag:SS} we obtain the maps $\Phi^*: {^T}E_2^{*,*}\to {^K}E_2^{*,*}$ and $\Psi^*:{^U}E_2^{*,*}\to {^T}E_2^{*,*}$ of spectral sequences on $E_2$ pages. 

The strategy in \cite{Gu21} for computing $^UE$ is to compare its differentials with the differentials in $^TE$ and $^KE$, which we will explain in the next section. 

	\section{The cohomology of $K(\Zz,3)$ and  differentials of $^KE$ and $^TE$}\label{sec:K3}
Throughout this section, we write $R[x;k]$, $E_R[y;k]$ for the polynomial algebra and exterior algebra respectively, with coefficients in a commutative ring $R$, and with one generator of degree $k$. We also use the simplified notation $E[y;k]$ for $R=\Zz$.

	Let $C(3)$ be the (differential graded) $\Zz$-algebra
	\[\begin{split}
		C(3)&=\bigotimes_{p\neq 2\text{ prime}}\big(\bigotimes_{k\geq0}\Zz[y_{p,k};2p^{k+1}+2]\otimes E[x_{p,k};2p^{k+1}+1]\big)\\
		&\otimes\bigotimes_{k\geq0}\Zz[x_{2,k};2^{k+2}+1]\otimes\Zz[x_1;3].
	\end{split}\]
	For notational convenience, we set $y_{2,0}=x_1^2$ and $y_{2,k}=x_{2,k-1}^2$ for $k\geq 1$. The differential of $C(3)$ is given by
	\[d(x_{p,k})=py_{p,k},\ d(y_{p,k})=d(x_1)=0,\text{ for } p \text{ a prime and }k\geq0.\]
	Gu \cite{Gu21} showed  that the integral cohomology ring of $K(\Zz,3)$ is isomorphic to the cohomology of $C(3)$, which is dual to the DGA construction of Cartan in \cite{Car54}.
	
	Let $_pA$ denote the $p$-primary component of a commutative ring $A$, $p$ a prime. 
	For $R=\bigoplus_{i>0} R^i$ a graded commutative
	ring without unit, let $\hat R=R\oplus R^0$ be the (graded commutative) unital ring, where the degree zero summand $R^0=\Zz$ has a generator acting as the
	unit of the ring.
	
	\begin{thm}[{\cite[Proposition 2.14]{Gu21}}]\label{thm:cohomology of k(3)}
		The cohomology ring of $K(\Zz,3)$ is given by
		\[H^*(K(\Zz,3);\Zz)\cong \bigotimes_{p}\widehat{R_p}\otimes\Zz[x_1]/(x_1^2-y_{2,0}),\]
		where $p$ ranges over all prime numbers, and $R_p={_pH}^*(K(\Zz,3);\Zz)$ is the $\Zz/p$-subalgebra (without unit) of $C(3)\otimes\Zz/p$ generated by the elements of the form
		\[y_{p,I}:=\frac{1}{p}d(x_{p,i_1}\cdots x_{p,i_m})=\sum_{j=1}^m(-1)^{j-1}x_{p,i_1}\cdots y_{p,i_j}\cdots x_{p,i_m},\]
		for any ordered sequence $I=(i_1,\dots,i_m)$ of nonnegative integers
		$i_1>\cdots>i_m$. For $I=(i)$, we simply write $y_{p,i}$ for $y_{p,I}$.
	\end{thm}
	
	\begin{rem}
	In \cite[Expos\'e 11]{Car54} a differential graded algebra $A(\pi,n)$, which is dual to $C(3)$ for $n=3$ and $\pi=\Zz$, is constructed
	to calculate the integral homology of $K(\pi,n)$ for any positive integer $n$ and a finitely generated abelian group $\pi$. \cite{Gu21} first gives an explicit formula for $H^*(K(\Zz,3);\Zz)$ in terms of generators and relations.
	\end{rem}

For $p$ a prime, the mod $p$ cohomology ring of $K(\Zz,3)$ is also given by the cohomology of the differential graded $\Zz/p$-algebra $(C(3)\otimes\Zz/p,d)$, which has an easier description (see \cite[(2.13)]{Gu21}). 
		Furthermore, the Steenrod algebra structure of $H^*(K(\Zz,3);\Zz/p)$ can be determined from the fact that $x_1$ (resp. $x_{p,k}$) is the the transgression image of $v$ (resp. $v^{p^{k+1}}$) $\in H^*(BS^1;\Zz/p)\cong\Zz/p[v]$ in the mod $p$ Serre spectral sequence $^KE$, using Kudo's transgression theorem \cite{Kud56}.
		In particular, for $p=2$, we have the following analogue of \cite[Proposition 2.4]{Gu21b}. 
	\begin{prop}\label{prop:mod 2 cohomology of k(3)}
$H^*(K(\Zz,3);\Zz/2)\cong\bigotimes_{k\geq0}\Zz/2[x_{2,k};2^{k+2}+1]\otimes\Zz/2[x_1;3]$.
		Moreover, $Sq^1(x_1)=0$, $Sq^1(x_{2,k+1})=x_{2,k}^2$, 
		$x_{2,k}=Sq^{2^{k+1}}Sq^{2^k}\cdots Sq^{2}(x_1)$, $k\geq 0$.
	\end{prop}

	The differentials of the integral spectral sequence $^KE$ are described in \cite{Gu21} as follows. Here we only list a part of the results in {\cite[Corollary 2.16]{Gu21}}, which will be sufficient for our purposes.
	
	\begin{prop}[{\cite[Corollary 2.16]{Gu21}}]\label{prop:differentials K}
		The higher differentials of $^KE^{*,*}_*$ satisfy
		\[\begin{split}
			&d_3(v)=x_1,\\
			&d_{2p^{k+1}-1}(p^kx_1v^{lp^e-1})=v^{lp^e-1-(p^{k+1}-1)}y_{p,k},\ e>k\geq 0,\ \gcd(l,p)=1,\\
			&d_r(x_1)=d_r(y_{p,k})=0\text{ for all }r\geq2,\,k\geq0,
		\end{split}\]
	and the Leibniz rule.
		Here $^KE^{3,2(lp^e-1)}_{2p^{k+1}-1}\subset {^KE}^{3,2(lp^e-1)}_2$ is generated by $p^kx_1v^{lp^e-1}$.
	\end{prop}

To describe the differential of $^TE^{*,*}_*$, we first rewrite  $H^*(BT^n;\Zz)=\Zz[v_1,\dots,v_n]$ as $\Zz[v_1-v_n,\dots,v_{n-1}-v_n,v_n]$. Then $^TE^{*,*}_*$ is generated by elements of the form $\xi\cdot f(v_n)$, where $\xi\in H^*(K(\Zz,3);\Zz)$ and
$f(v_n)=\sum_ia_iv_n^i$ is a polynomial in $v_n$ with coefficients $a_i\in\Zz[v_1-v_n,\dots,v_{n-1}-v_n]$. Let $v_j'=v_j-v_n$ for $1\leq j\leq n-1$, and write 
$a_i=\sum_{t_1,\dots,t_{n-1}}k_{i,t_1,\dots,t_{n-1}}(v'_1)^{t_1}\cdots (v_{n-1}')^{t_{n-1}}$, $k_{i,t_1,\dots,t_{n-1}}\in\Zz$.
\begin{prop}[{\cite[Propositions 3.2-3.6]{Gu21}}]\label{prop:T}
In the above notation, the differential $^Td_r$ is given by 
\[^Td_r(\xi\cdot f(v_n))=\sum_i\sum_{t_1,\dots,t_{n-1}} {^K}d_r(k_{i,t_1,\dots,t_{n-1}}\xi v_n^i)\cdot(v'_1)^{t_1}\cdots (v_{n-1}')^{t_{n-1}},\]
where ${^K}d_r(k_{i,t_1,\dots,t_{n-1}}\xi v_n^i)$ is simply ${^K}d_r(k_{i,t_1,\dots,t_{n-1}}\xi v^i)$ with $v$ replaced by $v_n$. 
In particular, the differential $^Td_3^{0,*}$ is given by the ``formal devergence''
\[\nabla=\sum_{i=1}^n(\partial/\partial v_i):H^*(BT^n;\Zz)\to H^{*-2}(BT^n;\Zz)\]
in such a way that $^Td^{0,*}_3=\nabla(-)\cdot x_1$.
Moreover, the spectral sequence $^TE^{*,*}_*$ degenerates at $^TE^{0,*}_4$, i.e.,
\[
{^TE}^{0,*}_\infty={^T}E^{0,*}_4=\ker{^Td_3^{0,*}}=\Zz[v_1',\dots,v_{n-1}'].
\] 
\end{prop}
Since the operator $\nabla$ preserves symmetric polynomials, for $f\in H^*(BU(n);\Zz)\cong\Lambda_n$ and $\xi\in H^*(K(\Zz,3);\Zz)$, we have the following corollary of Proposition \ref{prop:T} by the Leibniz rule.
\begin{cor}[{\cite[Corollary 3.10]{Gu21}}]\label{cor:U^d_3}
	 $^Ud_3(f\xi)=\nabla(f)x_1\xi$.
\end{cor}
In particular, the fact that $\nabla(\sigma_k)=\sum_{i=1}^n\frac{\partial\sigma_k}{\partial v_i}=(n-k+1)\sigma_{k-1}$ shows that for $c_k\in H^*(BU(n);\Zz)\cong {^U}E_3^{0,*}$, $1\leq k\leq n$, we have  (\cite[Corollary 3.4]{Gu21}) 
\begin{equation}\label{eq:nabla}
^Ud_3(c_k)=(n-k+1)c_{k-1}x_1.
\end{equation}

\begin{cor}\label{cor:E_4}
	$^UE_4^{0,*}\cong K_n$.
\end{cor}
\begin{proof}
This comes from the fact that ${^U}E^{0,*}_3={^U}E^{0,*}_2=H^*(BU(n);\Zz)\cong \Lambda_n$, ${^U}E^{0,*}_4=\ker {^U}d_3^{0,*}$, where $^Ud_3^{0,*}=\nabla(-)\cdot x_1$ by Corollary \ref{cor:U^d_3}.
\end{proof}

\section{Proof of Theorem \ref{thm:K_4}}\label{sec:proof K_4}
The proof of Theorem \ref{thm:K_4} is essentially a $K_4$ analogue of the proof of Lemma 3.2 for $K_3$ in \cite{Vezz00} (also Theorem 14.2 in \cite{Vis07}).

Firstly, we give the explicit expression of $\alpha_i$ in $\Lambda_4^i$. Let \begin{equation}\label{eq:alpha}
	\begin{split}
		\alpha_2&=8\sigma_2-3\sigma_1^2,\\
		\alpha_3&=8\sigma_3-4\sigma_1\sigma_2+\sigma_1^3,\\
		\alpha_4&=4\sigma_4-\sigma_1\sigma_3+43\sigma_2^2-32\sigma_1^2\sigma_2+6\sigma_1^4,\\
		\alpha_6&=27\sigma_1^2\sigma_4+27\sigma_3^2-9\sigma_1\sigma_2\sigma_3-72\sigma_2\sigma_4+2\sigma_2^3.
	\end{split}
\end{equation}
A direct calculation shows that $\nabla(\alpha_i)=0$ for all $\alpha_i$ above, and \[64\alpha_6=97\alpha_2^3+27\alpha_3^2-144\alpha_2\alpha_4.\]

Consider the ring homomorphism 
\[\ker\nabla\cong\Zz[v_1-v_4,v_2-v_4,v_{3}-v_4]\hookrightarrow\Zz[v_1,\dots,v_4]\to\Zz[v_1,\dots,v_4]/(\sigma_1).\]
It is easy to see that this map is injective. After tensoring with $\Zz[\frac{1}{2}]$, this map becomes an isomorphism, and the inverse is obtained by sending $v_i$ to $v_i-\frac{\sigma_1}{4}$. The commutative diagram below shows that $K_4\otimes\Zz[\frac{1}{2}]\hookrightarrow\Zz[\frac{1}{2}][\sigma_2,\sigma_3,\sigma_4]$.  
\[
\xymatrix{
	K_4\otimes\Zz[\frac{1}{2}]\ar[r]\ar@{^{(}->}[d]&\Zz[\frac{1}{2}][\sigma_2,\sigma_3,\sigma_4]\ar@{^{(}->}[d]\\
	\ker\nabla\otimes\Zz[\frac{1}{2}]\ar@{=}[r]&\Zz[\frac{1}{2}][v_1,\dots,v_4]/(\sigma_1)}
\]
On the other hand, the inclusion $\Zz[\frac{1}{2}][\alpha_2,\alpha_3,\alpha_4]\hookrightarrow \Zz[\frac{1}{2}][\sigma_2,\sigma_3,\sigma_4]$ is an isomorphism, so in fact we have 
\[\Zz[\frac{1}{2}][\alpha_2,\alpha_3,\alpha_4]\cong K_4\otimes\Zz[\frac{1}{2}]\cong\Zz[\frac{1}{2}][\sigma_2,\sigma_3,\sigma_4].\]
Hence if $f\in K_4$ then $2^sf\in \Zz[\alpha_2,\alpha_3,\alpha_4]$ for sufficient large $s$, and  the following lemma shows that
$\alpha_2$, $\alpha_3$, $\alpha_4$, $\alpha_6$ generate $K_4$ by descending induction on $s$.
\begin{lem}
	For $f\in K_4$, if $2f\in \Zz[\alpha_2,\alpha_3,\alpha_4,\alpha_6]$, then $f\in\Zz[\alpha_2,\alpha_3,\alpha_4,\alpha_6]$ too.
\end{lem}
\begin{proof}
	The images of $\alpha_2$, $\alpha_3$, $\alpha_4$, $\alpha_6$ in $\Zz/2[\sigma_1,\dots,\sigma_4]$ is $\sigma_1^2$, $\sigma_1^3$, $\sigma_1\sigma_3+\sigma_2^2$, $\sigma_1^2\sigma_4+\sigma_3^2+\sigma_1\sigma_2\sigma_3$ respectively. The ideal of relations between these four polynomials
	is generated by $(\sigma_1^2)^3+(\sigma_1^3)^2$, which is the image of $97\alpha_2^3+27\alpha_3^2-144\alpha_2\alpha_4=64\alpha_6$. Hence if we write $2f=p(\alpha_2,\alpha_3,\alpha_4,\alpha_6)$ for some integral polynomial $p$, then there are two integral polynomials $q$ and $r$ such that
	\[p(\alpha_2,\alpha_3,\alpha_4,\alpha_6)=2q(\alpha_2,\alpha_3,\alpha_4,\alpha_6)+64\alpha_6\cdot r(\alpha_2,\alpha_3,\alpha_4,\alpha_6).\]
	Dividing by $2$ we get $f\in \Zz[\alpha_2,\alpha_3,\alpha_4,\alpha_6]$.
\end{proof}

Consider the surjective ring homomorphism $\Zz[y_2,y_3,y_4,y_6]\to K_4$ that sends $y_i$ to $\alpha_i$. Let $I$ be its kernel. Then the following lemma is an equivalent statement of Theorem \ref{thm:K_4}.
\begin{lem}
	In the above notation, $I=(64y_6-97y_2^3-27y_3^2+144y_2y_4)$.
\end{lem}
\begin{proof}
	Let $\yy=64y_6-97y_2^3-27y_3^2+144y_2y_4$. We know that $\yy\in I$. Note that  \[\Zz[\frac{1}{2}][y_2,y_3,y_4,y_6]/(\yy)=\Zz[\frac{1}{2}][y_2,y_3,y_4],\] 
	which is isomorphic to $K_4\otimes\Zz[\frac{1}{2}]$ as we saw above. This implies that if $f\in I$, then some multiple $2^rf$ is in $(\yy)$. However, since $2$ dose not divide $\yy$ in the UFD $\Zz[y_2,y_3,y_4,y_6]$, $f$ is actually in $(\yy)$, and we get the desired equality in the lemma.
\end{proof}

We end this section with an algebraic result, which will be used in Section \ref{sec:^UE}.
\begin{prop}\label{prop:order}
	$\Zz/4[\alpha_4,\alpha_6]$ is a subring of $\coker\nabla$, where $\nabla$ acts on $\Lambda_4$.
\end{prop}
\begin{proof}
	Since $\alpha_4$, $\alpha_6\in K_4$ and $\nabla(\sigma_1)=4$, $\nabla(\sigma_1f)=4f$ for any $f\in\Zz[\alpha_4,\alpha_6]$. So the order of $f$ in $\coker\nabla$ divides $4$. If its order is not $4$, there would exist $g\in\Lambda_4$ such that $\nabla(g)=2f$, and then $\sigma_1f-2g\in K_4$. Let $\rho:\Lambda_4\to\Lambda_4\otimes\Zz/2$ be the mod $2$ map, and write $\w h=\rho(h)$ for an element $h\in\Lambda_4$. Then $\rho(\sigma_1f-2g)=\sigma_1\w f\in K_4\otimes\Zz/2=\Zz/2[\w\alpha_2,\w\alpha_3,\w\alpha_4,\w\alpha_6]$, where $\w\alpha_2=\sigma_1^2$, $\w\alpha_3=\sigma_1^3$, $\w\alpha_4=\sigma_1\sigma_3+\sigma_2^2$, $\w\alpha_6=\sigma_1^2\sigma_4+\sigma_3^2+\sigma_1\sigma_2\sigma_3$. This means that there is a polynomial $p$ over $\Zz/2$  such that 
	$\sigma_1\w f=p(\w\alpha_2,\w\alpha_3,\w\alpha_4,\w\alpha_6)$ in $\Zz/2[\sigma_1,\dots,\sigma_4]$. We may assume that $\w f$ and $p$ are homogeneous polynomials in $\Zz/2[\sigma_1,\dots,\sigma_4]$, so that $\deg(p)=\deg(\w f)+1$, an odd number, which implies that $\w\alpha_3$ divides $p$. However, since $\Zz/2[\sigma_1,\dots,\sigma_4]$ is a UFD and $\sigma_1$ dose not divide $\w\alpha_4$ and $\w\alpha_6$, $\w\alpha_3=\sigma_1^3$ can not divide $\sigma_1\w f$. This gives a contradiciton. So the order of $f$ must be $4$, and the proposition follows immediately.
\end{proof}

\section{On the homomorphism from $H^*(K(\Zz,3))$ to $H^*(BPU(2))$}\label{sec:hom}
In this section, we study the homomorphism $\chi^*:H^*(K(\Zz,3))\to H^*(BPU(2))$ induced by the fibration map $\chi:BPU(2)\to K(\Zz,3)$ defined in \eqref{eq:map chi}. Throughout this section, let $\rho:H^*(-;\Zz)\to H^*(-;\Zz/2)$ denote the mod $2$ reduction map on cohomology.

Since $PU(2)=SO(3)$, the following result is a special case of \cite{Brow82}.
\begin{thm}\label{thm:SO3}
	Let $w_i\in H^i(BSO(3);\Zz/2)$ be the $i$th universal Stiefel-Whitney class of $BSO(3)$, and $p_1\in H^4(BSO(3);\Zz)$ the first Pontryagin class. Then 
	\[\begin{split}
	&H^*(BSO(3);\Zz/2)=\Zz/2[w_2,w_3],\ w_3=Sq^1(w_2), \text{ and }\\
	&H^*(BSO(3);\Zz)=\Zz[p_1,W_3]/(2W_3),\ \rho(p_1)=w_2^2,\ \rho(W_3)=w_3. 
 	\end{split}\] 
Here $W_3$ is the degree-$3$ integral Stiefel-Whitney class.
\end{thm}

\begin{thm}\label{thm:hom mod 2}
	In the notation of Proposition \ref{prop:mod 2 cohomology of k(3)}, the mod $2$ cohomology homomorphism $\chi^*:H^*(K(\Zz,3);\Zz/2)\to H^*(BPU(2);\Zz/2)$ satisfies
	\begin{gather*}
		\chi^*(x_1)=w_3,\quad \chi^*(x_{2,0})=w_2w_3,\\ \chi^*(x_{2,k})\equiv w_2^{2^{k+1}-1}w_3\mod\,(w_3^3),\ \text{ for }k\geq1.
	\end{gather*}
\end{thm}
\begin{proof}
	First we show that $\chi^*(x_1)=w_3$. This follows from the computation in \cite{Gu21} that $H^3(BPU(n);\Zz)=\Zz/n$ is generated by $\chi^*(x_1)$.	Here we give a short explanation to this. Since $^UE_{2}^{0,*}={^UE}_{3}^{0,*}=H^*(BU(n);\Zz)\cong \Zz[c_1,\dots,c_n]$ and $^Ud_3^{0,*}$ is $\nabla(-)\cdot x_1$, it follows that $^Ud_3^{0,2}(c_1)=nx_1\in {^UE}_{3}^{3,0}$. Therefore, for degree reasons \[^UE_{4}^{3,0}={^UE}_{\infty}^{3,0}=H^3(BPU(n);\Zz)=\Zz/n\{\chi^*(x_1)\}.\] 
	
	Now we give the formula for $\chi^*(x_{2,k})$ by induction on $k$, starting with $\chi^*(x_{2,0})$. Since $Sq^1(x_{2,0})=x_1^2$ by Proposition \ref{prop:mod 2 cohomology of k(3)}, we have $Sq^1\chi^*(x_{2,0})=\chi^*Sq^1(x_{2,0})=w_3^2$. Hence $\chi^*(x_{2,0})=w_2w_3$, the only nonzero element of degree $5$. Inductively, suppose that $\chi^*(x_{2,k-1})$ is congruent to $w_2^{2^{k}-1}w_3$ modulo the ideal generated by $w_3^3$. Then since $Sq^1(x_{2,k})=x_{2,k-1}^2$ by Proposition \ref{prop:mod 2 cohomology of k(3)}, we have \[Sq^1\chi^*(x_{2,k})=\chi^*Sq^1(x_{2,k})\equiv w_2^{2^{k+1}-2}w_3^2 \mod (w_3^3).\] 
	Since $Sq^1$ preserves $(w_3^3)$ and in $H^*(BSO(3);\Zz/2)/(w_3^3)$, $w_2^{2^{k+1}-1}w_3$ is the only degree-$(2^{k+2}+1)$ element mapped by $Sq^1$ to $w_2^{2^{k+1}-2}w_3^2$, it follows that $\chi^*(x_{2,k})\equiv w_2^{2^{k+1}-1}w_3$ mod $(w_3^3)$, and the induction step is finished. 
\end{proof}

\begin{cor}\label{cor:hom}
In the notation of Theorem \ref{thm:cohomology of k(3)} and Theorem \ref{thm:SO3}, the homomorphism $\chi^*:H^*(K(\Zz,3);\Zz)\to H^*(BPU(2);\Zz)$ satisfies
$\chi^*(x_1)=W_3$, and $\chi^*(y_{2,k})\equiv p_1^{2^k-1}W_3^2\mod(W_3^6)$ for $k\geq 0$.
\end{cor}

\begin{proof}
The formula for $\chi^*(x_1)$ is obvious. From Section \ref{sec:K3}, we know that  $\rho(y_{2,0})=x_1^2$ and $\rho(y_{2,k})=x_{2,k-1}^2$ for $k\geq 1$. 
Then, by Theorem \ref{thm:hom mod 2} we have $\rho\chi^*(y_{2,k})=\chi^*\rho(y_{2,k})\equiv w_2^{2^{k+1}-2}w_3^2\mod(w_3^6)$ for $k\geq 0$. This implies that $\chi^*(y_{2,k})\equiv p_1^{2^k-1}W_3^2\mod(W_3^6)$ since $p_1^{2^k-1}W_3^2$ is the only torsion element in $H^*(BPU(2);\Zz)/(W_3^6)$ mapped by $\rho$ to $w_2^{2^{k+1}-2}w_3^2\mod(w_3^6)$.
\end{proof}

\begin{prop}\label{prop:nonzero}
	Let $n>0$ be an interger such that $2\mid n$. Then for any $i,j,k\geq 0$,
	
	(a) $x_1^ix_{2,k}^j\neq 0$ in $H^*(BPU(n);\Zz/2)$.
	
	(b) $x_1^iy_{2,k}^j\neq 0$ in $H^*(BPU(n);\Zz)$.
	
	Here we abuse the notations $x_1$, $x_{2,k}$, $y_{2,k}$, to denote their images under the cohomology ring homomorphism induced by the fibration map $BPU(n)\to K(\Zz,3)$.
\end{prop}
\begin{proof}
	Since $2\mid n$, there is a diagonal map of matrices 
	\[U(2)\to U(n),\quad A\mapsto \begin{bmatrix}
		A&0&\cdots&0\\
		0&A&\cdots&0\\
		\vdots&\vdots&\ddots&\vdots\\
		0&\cdots&\cdots&A
	\end{bmatrix},\] which passes to $PU(2)\to PU(n)$.
	These diagonal maps induce maps $\Delta:BU(2)\to BU(n)$ and $\bar\Delta:BPU(2)\to BPU(n)$, and a commutative diagram of fibrations
	\begin{equation}\label{diag:diagonal}
		\begin{gathered}
			\xymatrix{
				BU(2)\ar[r]\ar^{\Delta}[d]&BPU(2)\ar[r]\ar[d]^{\bar\Delta}&K(\Zz,3)\ar[d]^{=}\\
				BU(n)\ar[r]&BPU(n)\ar[r]&K(\Zz,3)}
		\end{gathered}
	\end{equation}
Since the map $\bar\Delta^*$ takes the elements $x_1$, $x_{2,k}$, $y_{2,k}$, to themself, the corollary follows from Theorem \ref{thm:hom mod 2} and Corollary \ref{cor:hom}.
\end{proof}

\section{On the spectral sequence $^UE$}\label{sec:^UE}
\begin{prop}\label{prop:K_4}
	In  the integral spectral sequence $^UE$ associated to $BU(4)\xrightarrow{\pi} BPU(4)\xrightarrow{\chi} K(\Zz,3)$, we have	$K_4\cong {^U}E^{0,*}_4={^U}E^{0,*}_\infty$. Moreover, $\Zz/4[\alpha_4,\alpha_6]$ is a submodule of ${^U}E^{3,*}_\infty$.
\end{prop}
\begin{proof}
By Corollary \ref{cor:E_4}, $K_4\cong {^U}E^{0,*}_4$.  Since the differentials in the spectral sequence $^UE$ satisfy the Leibniz rule and $K_4$ is generated by $\alpha_2$, $\alpha_3$, $\alpha_4$, $\alpha_6$, it suffices to show that these generators survive to $^UE_\infty$. For $\alpha_2$, $\alpha_3$, $\alpha_4$, this was proved in \cite{Gu21}, so we only consider $\alpha_6\in {^U}E^{0,12}_4$. 

The differentials in $^UE$ have the form
$d_r:{^UE}^{s,t}_r\to{^UE}^{s+r,t-r+1}_r$. Hence from Theorem \ref{thm:cohomology of k(3)} we see that the only possible nontrivial differentials $d_r$, $r\geq 4$, originating at ${^U}E^{0,12}_*$, are $d_9:{^U}E^{0,12}_9\to {^U}E^{9,4}_9$, $d_{11}:{^U}E^{0,12}_{11}\to {^U}E^{11,2}_{11}$ and $d_{13}:{^U}E^{0,12}_{13}\to {^U}E^{13,0}_{13}$. Note that ${^U}E^{9,4}_2\cong x_1^3\cdot H^4(BU(4);\Zz)$, ${^U}E^{11,2}_2\cong \Zz/3\{x_1x_{3,0}c_1\}$ and
	${^U}E^{13,0}_2\cong \Zz/2\{x_1y_{2,1}\}$ by Theorem \ref{thm:cohomology of k(3)}. 
	
	From Proposition \ref{prop:nonzero} we know that $x_1y_{2,1}$ is a permanent cycle, so it can not be hitten by any differential in $^UE$.
	Furthermore, from \eqref{eq:nabla} we see that $d_3(x_{3,0}c_1^2)\equiv -x_1x_{3,0}c_1$ mod $3$, so ${^U}E^{11,2}_{4}=0$. 
	Hence the proposition will follow once we prove that $^UE_4^{9,4}=0$. This is an easy calculation using the following commutative diagram
	\[\xymatrix{\Lambda_4^3\otimes\Zz/2\ar[r]^{\nabla_{\Zz/2}}\ar^{\cong}[d]&\Lambda_4^2\otimes\Zz/2\ar[r]^{\nabla_{\Zz/2}}\ar[d]^{\cong}&\Lambda_4^1\otimes\Zz/2\ar@{^{(}->}[d]\\
		{^UE}_3^{6,6}\ar[r]^{d_3^{6,6}}&{^UE}_3^{9,4}\ar[r]^{d_3^{9,4}}&{^UE}_3^{12,2}}
		\]
The right vertival map is only an injection since besides $x_1^4\cdot H^2(BU(4);\Zz)$, $y_{5,0}\cdot H^2(BU(4);\Zz)$ also contributes to ${^UE}_3^{12,2}$, but this does not effect our computation.	It can be shown that $\ker d_3^{9,4}=\Zz/2\{x_1^3c_1^2\}$ and $d_3^{6,6}(x_1^2c_2c_1)=x_1^3c_1^2$. Then we have $^UE_4^{9,4}=\ker d_3^{9,4}/\im d_3^{6,6}=0$.

Using Proposition \ref{prop:order}, we see  that $x_1f$ in $^UE_4^{3,*}$ has order $4$ for any $f\in\Zz[\alpha_4,\alpha_6]\subset H^*(BU(4);\Zz)$. Since $\alpha_4$, $\alpha_6$ and $x_1$ are all permanent cycles, so is $x_1f$, and the second statement follows.
\end{proof}
Proposition \ref{prop:K_4} says that there exist elements $\bar\alpha_i\in H^{2i}(BPU(4);\Zz)$, $i=2,3,4,6$, such that  $\pi^*(\bar\alpha_i)=\alpha_i\in H^{2i}(BU(4);\Zz)$.
Let $\Delta:BU(2)\to BU(4)$, $\bar\Delta:BPU(2)\to BPU(4)$ be the maps defined in the proof of Proposition \ref{prop:nonzero}. We consider the images of $\alpha_i$ and $\bar\alpha_i$ under the maps $\Delta^*$ and $\bar\Delta^*$ respectively. 

Let $c_i$ and $c_i'$ be the $i$th universal Chern classes of $BU(4)$ and $BU(2)$ respectively. 
Since the map $\Delta$ factors through the diagonal map $BU(2)\to BU(2)\times BU(2)$, using the Whitney sum formula and the functorial property of total Chern classes we immediately get
\[
	\Delta^*(1+c_1+c_2+c_3+c_4)=(1+c_1'+c_2')^2.
\]
In other words, 
\begin{equation}\label{eq:chern class}
\begin{gathered}
\Delta^*(c_1)=2c_1',\quad \Delta^*(c_2)=c_1'^2+2c_2',\\
 \Delta^*(c_3)=2c_1'c_2',\quad \Delta^*(c_4)=c_2'^2.
\end{gathered}
\end{equation}

\begin{prop}\label{prop:alpha_i}
The elements $\bar\alpha_i\in H^{2i}(BPU(4);\Zz)$, $i=2,3,4,6$, can be chosen such that $\bar\Delta^*(\bar\alpha_4)=11p_1^2$, and $\bar\Delta^*(\bar\alpha_i)\equiv 0\mod2$ for $i=2,3,6$.
\end{prop}
\begin{proof}
	It follows from \eqref{eq:alpha} and \eqref{eq:chern class} that $\Delta^*(\alpha_i)\equiv 0$ mod $2$ for $i=2,3,6$, and $\Delta^*(\alpha_4)=11(c_1'^2-4c_2')^2$. 
	On the other hand, for the spectral sequence $^UE$ associated to $BU(2)\xrightarrow{\pi} BPU(2)\xr{\chi}K(\Zz,3)$, an easy calculation shows that ${^UE_\infty^{0,4}}={^UE_4^{0,4}}\cong K_2=\Zz[\sigma_1^2-4\sigma_2]\subset\Lambda_2$. 
	This means that for $p_1\in H^4(BSO(3);\Zz)=H^4(BPU(2);\Zz)$, $\pi^*(p_1)=\pm (c_1'^2-4c_2')$. 
	
Recall from Theorem \ref{thm:SO3} that 
\begin{gather*}
H^4(BPU(2);\Zz)\cong \Zz\{p_1\},\ \  H^6(BPU(2);\Zz)\cong \Zz/2\{W_3^2\},\\ H^8(BPU(2);\Zz)\cong \Zz\{p_1^2\},\ \ H^{12}(BPU(2);\Zz)\cong \Zz\{p_1^3\}\oplus\Zz/2\{W_3^4\}.
\end{gather*} 
Since $\pi^*\bar\Delta^*=\Delta^*\pi^*$, combining the above computations, we obtain $\bar\Delta^*(\bar\alpha_4)=11p_1^2$ and $\bar\Delta^*(\bar\alpha_2)\equiv 0$ mod $2$. We have seen that $\bar\Delta^*(x_1)=W_3$, hence by adding $x_1^2$ to $\bar\alpha_3$ or $x_1^4$ to $\bar\alpha_6$ if necessary, we can also obtain $\bar\Delta^*(\bar\alpha_3),\,\bar\Delta^*(\bar\alpha_6)\equiv 0$ mod $2$.
\end{proof}

\section{On $H^*(BPU(4);\Zz/2)$}\label{sec:on the mod 2 cohomology}
In this section we prove some facts about the generators $y_i$ of $H^*(BPU(4);\Zz/2)$ in Theorem \ref{thm:toda}. 
\begin{prop}\label{prop:odd generator}
	In the notation of Proposition \ref{prop:mod 2 cohomology of k(3)}, $\chi^*(x_1)=y_3$, $\chi^*(x_{2,0})=y_5$, $\chi^*(x_{2,1})=y_9$ or $y_3^3+y_9$.
\end{prop}

\begin{proof}
	From Proposition \ref{prop:nonzero} we know that $\chi^*(x_1)$, $\chi^*(x_{2,0})$, $\chi^*(x_{2,1})\neq0$. Since $\deg(x_1)=3$ and $y_3$ is the only nonzero element in $H^*(BPU(4);\Zz/2)$, $\chi^*(x_1)=y_3$. Similarly, $\chi^*(x_{2,0})=y_5$. 
	
	From Theorem \ref{thm:toda} one easily sees that $H^9(BPU(4);\Zz/2)=\Zz/2\{y_3^3,y_9\}$. Since $Sq^1(x_{2,1})=x_{2,0}^2$ by Proposition \ref{prop:mod 2 cohomology of k(3)}, it follows that $Sq^1\chi^*(x_{2,1})=\chi^*(x_{2,0}^2)=y_5^2$.  Hence $\chi^*(x_{2,1})$ takes one of the forms in the proposition since 
	$Sq^1(y_3)=Sq^1\chi^*(x_1)=\chi^*Sq^1(x_1)=0$.
\end{proof}

Recall from Proposition \ref{prop:mod 2 cohomology of k(3)} that $Sq^1(x_1)=0$, $Sq^1(x_{2,0})=x_1^2$, $Sq^1(x_{2,1})=x_{2,0}^2$, $Sq^2(x_1)=x_{2,0}$, and $Sq^4(x_{2,0})=x_{2,1}$. Then we have the following corollary of Proposition \ref{prop:odd generator}.

\begin{cor}\label{cor:steenrod square for odd generator}
In $H^*(BPU(4);\Zz/2)$, $Sq^1(y_3)=0$, $Sq^1(y_5)=y_3^2$, $Sq^1(y_9)=y_5^2$, $Sq^2(y_3)=y_5$, $Sq^4(y_5)=y_9$ or $y_3^3+y_9$.	
\end{cor}

Before proceeding further, let us review a result  in \cite{Tod87}. Let $PSO(2n)$ be the  quotient group of $SO(2n)$ by the center $\Zz/2=\{\pm 1\}$, and let $\phi:BSO(2n)\to BPSO(2n)$ be the induced map of classifying spaces. It is known that $PU(4)$ is homeomorphic to $PSO(6)$, so there is an induced homomorphism \[\phi^*:H^*(BPU(4);\Zz/2)\to H^*(BSO(6);\Zz/2)\cong\Zz[w_2,\dots,w_6],\]
where $w_i$ is the $i$th universal Stiefel-Whitney class of $BSO(2n)$.
\begin{prop}[{\cite[Proposition 3.7 and 4.5]{Tod87}}]\label{prop:image in BSO(6)}
	$\phi^*(y_2)=0$, $\phi^*(y_3)=w_3$, $\phi^*(y_5)=\bar w_5$, $\phi^*(y_8)=\bar w_4^2+w_3^2 w_2$, $\phi^*(y_{12})=\bar w_6^2+\bar w_5^2 w_2$, where $\bar w_4=w_4+w_2^2$, $\bar w_5=w_5+w_2w_3$, $\bar w_6=w_6+w_2w_4$, and $\phi^*(y_9)\equiv w_3w_6+w_4w_5\mod (w_2)$.
\end{prop}
(See \cite[p. 91]{Tod87} for the formulas of $\bar w_i$, $i=4,5,6$.)
\begin{prop}\label{prop:w_i}
	In Proposition \ref{prop:image in BSO(6)},  $\phi^*(y_9)=w_3w_6+w_4w_5+w_2^2w_5+w_2^3w_3$. Moreover, $Sq^4(y_5)$, i.e. $\chi^*(x_{2,1})$, is actually $y_3^3+y_9$ in Corollary \ref{cor:steenrod square for odd generator}.
\end{prop}
\begin{proof}
Applying the Wu formula for Stiefel-Whitney classes:
\[Sq^i(w_k)=\sum_{j=0}^i\binom{k-j-1}{i-j}w_{k+i-j}w_{j},\]
 and the Cartan formula for Steenrod squares we get 
\[\begin{split}
Sq^4(\bar w_5)&=Sq^4(w_5)+Sq^4(w_2w_3)\\
&=w_3w_6+w_4w_5+Sq^1(w_2)Sq^3(w_3)+Sq^2(w_2)Sq^2(w_3)\\
&=w_3w_6+w_4w_5+w_3^3+w_2^2w_5+w_2^3w_3.
\end{split}
\]
By Corollary \ref{cor:steenrod square for odd generator} and Proposition \ref{prop:image in BSO(6)}, $Sq^4(\bar w_5)=Sq^4\phi^*(y_5)=\phi^*Sq^4(y_5)=\phi^*(y_9)$ or $\phi^*(y_9)+w_3^3$, where $\phi^*(y_9)\equiv w_3w_6+w_4w_5$ mod $(w_2)$. Combining this with the above expression for $Sq^4(\bar w_5)$, we obtain $\phi^*(y_9)=w_3w_6+w_4w_5+w_2^2w_5+w_2^3w_3$. 
This computation also shows that $Sq^4(y_5)=y_3^3+y_9$, which is the second statement.
\end{proof}

Let $\pi:BU(4)\to BPU(4)$ be the map in \eqref{eq:map pi}. The following result  in \cite{Tod87} gives the induced mod $2$ cohomology ring homomorphism $\pi^*$.
\begin{prop}[{\cite[(4.11)]{Tod87}}]\label{prop:toda image pi}
In $H^*(BU(4);\Zz/2)$,	$\pi^*(y_2)=c_1$, $\pi^*(y_8)=c_1c_3+c_2^2$,  $\pi^*(y_{12})=c_1^2c_4+c_3^2+c_1c_2c_3$.
\end{prop}

Let $\Delta:BU(2)\to BU(4)$ and $\bar\Delta:BSO(3)\cong BPU(2)\to BPU(4)$ be the maps defined in the proof of Proposition \ref{prop:nonzero}, and to avoid ambiguity of notations, let $w_i'$  denote the $i$th universal Stiefel-Whitney class of $BSO(3)$.

\begin{prop}\label{prop:even generator}
$Sq^1(y_2)=0$, $Sq^1(y_8)=y_3^3$, $Sq^1(y_{12})=y_3y_5^2$, and $\bar\Delta^*(y_2)=0$, $\bar\Delta^*(y_3)=w_3'$, $\bar\Delta^*(y_5)=w_2'w_3'$, $\bar\Delta^*(y_8)=w_2'^4+w_2'w_3'^2$, $\bar\Delta^*(y_9)=w_2'^3w_3'$, $\bar\Delta^*(y_{12})=w_2'^3w_3'^2$.	
\end{prop}
\begin{proof}
	Since $H^3(BPU(4);\Zz/2)\cong \Zz/2\{y_3\}$, the equation $Sq^1(y_2)=0$ follows  from the fact that $\phi^*Sq^1(y_2)=Sq^1\phi^*(y_2)=0\neq w_3=\phi^*(y_3)$. Let $\pi$ also denote the map $BU(2)\to BPU(2)$.
	Then $\pi^*\bar\Delta^*(y_2)=\Delta^*\pi^*(y_2)=\Delta^*(c_1)=0$ by Proposition \ref{prop:toda image pi} and \eqref{eq:chern class}. This gives $\bar\Delta^*(y_2)=0$ since $H^2(BPU(2);\Zz/2)\cong\Zz/2\{w_2'\}$ and $\pi^*(w_2')=c_1'\neq 0$.
	
	By Theorem \ref{thm:hom mod 2}, $\bar\Delta^*(y_3)=\bar\Delta^*\chi^*(x_1)=w_3'$, and $\bar\Delta^*(y_5)=\bar\Delta^*\chi^*(x_{2,0})=w_2'w_3'$. $\bar\Delta^*(y_9)=\bar\Delta^*(Sq^4(y_5)+y_3^3)=w_2'^3w_3'$ comes from the the second statement of Proposition \ref{prop:w_i} and the fact that $\bar\Delta^*Sq^4(y_5)=Sq^4\bar\Delta^*(y_5)=Sq^4(w_2'w_3')=w_2'^3w_3'+w_3'^3$.
	
We see from Proposition \ref{prop:image in BSO(6)} that $Sq^1\phi^*(y_8)=w_3^3=\phi^*(y_3^3)$. Hence $Sq^1(y_8)=y_3^3$ since $H^9(BPU(4);\Zz/2)=\Zz/2\{y_3^3,y_9\}$ and $\phi^*(y_9)\neq 0$.
	 Recall from Theorem \ref{thm:SO3} that $H^8(BPU(2);\Zz/2)\cong\Zz/2\{w_2'^4,w_2'w_3'^2\}$. Hence $\bar\Delta^*(y_8)=w_2'w_3'^2+aw_2'^4$, $a\in\Zz/2$, since $Sq^1\bar\Delta^*(y_8)=\bar\Delta^*Sq^1(y_8)=w_3'^3$.
	Furthermore, by Proposition \ref{prop:toda image pi}, $\pi^*(y_8)=c_1c_3+c_2^2\in H^*(BU(4);\Zz/2)$, and then $\pi^*\bar\Delta^*(y_8)=\Delta^*\pi^*(y_8)=c_1'^4$ by \eqref{eq:chern class}, which implies that the coefficient $a$ is $1$ since $\pi^*(w_3')=0$.
	
	Similarly, we have $Sq^1(y_{12})=ay_3y_5^2+by_5y_8\neq 0$ in $H^{13}(BPU(4);\Zz/2)\cong\Zz/2\{y_3y_5^2,\,y_5y_8\}$. Since $Sq^1(y_3y_5^2)=0$ and $Sq^1(y_5y_8)=y_3^2y_8+y_5y_3^3\neq 0$, it follows that $Sq^1(y_{12})=y_3y_5^2$.
	The relation $y_9^2+y_3^2y_{12}+y_5^2y_8=0$ in $H^*(BPU(4);\Zz/2)$ shows that
	\[\begin{split}
	&\bar\Delta^*(y_9^2+y_3^2y_{12}+y_5^2y_8)\\
	=&(w_2'^3w_3')^2+w_3'^2\bar\Delta^*(y_{12})+(w_2'w_3')^2(w_2'^4+w_2'w_3'^2)=0.
	\end{split}\]
	 Solving this equation we obtain $\bar\Delta^*(y_{12})=w_2'^3w_3'^2$.
\end{proof}

Let $\alpha_i\in H^*(BU(4);\Zz)$ and $\bar\alpha_i\in H^*(BPU(4);\Zz)$, $i=2,3,4,6$, be as in Section \ref{sec:^UE}, and let $\rho:H^*(-;\Zz)\to H^*(-;\Zz/2)$ denote the mod $2$ reduction map.

\begin{prop}\label{prop:Z_2 alpha_i}
 $\rho(\bar\alpha_2)=y_2^2$, $\rho(\bar\alpha_3)=y_2^3$, $\rho(\bar\alpha_4)=y_8+y_3y_5$, $\rho(\bar\alpha_6)=y_{12}+y_3y_9$.
\end{prop} 

\begin{proof}
By \eqref{eq:alpha}, $\alpha_2\equiv c_1^2$, $\alpha_3\equiv c_1^3$, $\alpha_4\equiv c_1c_3+c_2^2$, $\alpha_6\equiv c_1^2c_4+c_3^2+c_1c_2c_3$ mod $2$. 
By Proposition \ref{prop:toda image pi}, this means that $\pi^*\rho(\bar\alpha_2)=\rho(\alpha_2)=\pi^*(y_2^2)$, $\pi^*\rho(\bar\alpha_3)=\rho(\alpha_3)=\pi^*(y_2^3)$, $\pi^*\rho(\bar\alpha_4)=\rho(\alpha_4)=\pi^*(y_8)$, $\pi^*\rho(\bar\alpha_6)=\rho(\alpha_6)=\pi^*(y_{12})$ in $H^*(BU(4);\Zz/2)$. 

$\rho(\bar\alpha_2)=y_2^2$ is obvious since $y_2^2$ is the only nonzero element of $H^4(BPU(4);\Zz/2)$. $\rho(\bar\alpha_3)=y_2^3$ follows from the facts that  $\bar\Delta^*\rho(\bar\alpha_3)=0$ by Proposition \ref{prop:alpha_i} and  that $H^6(BPU(4);\Zz/2)\cong\Zz/2\{y_2^3,y_3^2\}$, $\bar\Delta^*(y_2)=0$, $\bar\Delta^*(y_3^2)=w_3'^2\neq 0$ by Proposition \ref{prop:even generator}. 

Proposition \ref{prop:alpha_i} also gives $\bar\Delta^*\rho(\bar\alpha_4)=w_2'^4$ because $\rho(p_1)=w_2'^2$.
Notice that $H^8(BPU(4);\Zz/2)\cong\Zz/2\{y_2^4,y_3y_5,y_8\}$. Hence from $\pi^*(y_2^4)=c_1^4$, $\pi^*\rho(\bar\alpha_4)=\pi^*(y_8)$, and  $\bar\Delta^*(y_3y_5)=w_2'w_3'^2$, $\bar\Delta^*(y_8)=w_2'^4+w_2'w_3'^2$ by Proposition \ref{prop:even generator}, it follows that $\rho(\bar\alpha_4)=y_8+y_3y_5$.

The fact that 
 $\pi^*\rho(\bar\alpha_6)=\pi^*(y_{12})=c_1^2c_4+c_3^2+c_1c_2c_3$ shows that $\rho(\bar\alpha_6)=y_{12}+ay_3y_9+by_3^4$, $a,b\in\Zz/2$, since 
 $H^{12}(BPU(4);\Zz/2)\cong\Zz/2\{y_2^6,y_3^4,y_3y_9,y_{12}\}$ and $\pi^*(y_2^6)=c_1^6$. 
Also, we have $\bar\Delta^*\rho(\bar\alpha_6)=0$ by Proposition  \ref{prop:alpha_i}, and $\bar\Delta^*(y_3^4)=w_3'^4$, $\bar\Delta^*(y_3y_9)=w_2'^3w_3'^2$, $\bar\Delta^*(y_{12})=w_2'^3w_3'^2$ by Proposition \ref{prop:even generator}. These together give the coefficients $a=1$ and $b=0$. The proof is completed.
\end{proof}
  
\section{Proof of Theorem \ref{thm:steenrod}}
The strategy of the proof is to compute the action of steenrod squares on the images of the generators $y_i$ under the map $\pi^*\oplus\phi^*\oplus\bar\Delta^*:$
\[H^*(BPU(4);\Zz/2)\to H^*(BU(4);\Zz/2)\oplus H^*(BSO(6);\Zz/2)\oplus H^*(BPU(2);\Zz/2).\]
As in Section \ref{sec:on the mod 2 cohomology}, we use $w_i'$ and $w_i$ to denote the $i$th universal Stiefel-Whitney class of $BSO(3)$ and $BSO(6)$ respectively.

First we consider $Sq^i$ on $\im\pi^*$.
\begin{prop}\label{prop:steenrod pi}
$Sq^2\pi^*(y_8)=0$, $Sq^2\pi^*(y_{12})=\pi^*(y_2y_{12})$, $Sq^4\pi^*(y_8)=\pi^*(y_2^2y_8+y_{12})$, $Sq^4\pi^*(y_{12})=\pi^*(y_2^2y_{12})$, $Sq^8\pi^*(y_{12})=\pi^*(y_8y_{12})$.
\end{prop}
\begin{proof}
Recall  that $\pi^*(y_2)=c_1$, $\pi^*(y_8)=c_1c_3+c_2^2$ and $\pi^*(y_{12})=c_1^2c_4+c_3^2+c_1c_2c_3$. Then the calculation can be carried out by using the Cartan formula for Steenrod squares and the Wu formula for the mod $2$ Chern classes:
\[Sq^{2i}(c_k)=\sum_{j=0}^i\binom{k-j-1}{i-j}c_{k+i-j}c_{j}.\]
For $H^*(BU(4);\Zz/2)\cong\Zz/2[c_1,\dots,c_4]$, we have  $Sq^2(c_1)=c_1^2$, $Sq^2(c_2)=c_1c_2+c_3$, $Sq^2(c_3)=c_1c_3$, $Sq^2(c_4)=c_1c_4$, $Sq^4(c_2)=c_2^2$, $Sq^4(c_3)=c_1c_4+c_2c_3$, $Sq^4(c_4)=c_2c_4$, $Sq^8(c_4)=c_4^2$.
Thus,
\[Sq^2\pi^*(y_8)=Sq^2(c_1)c_3+c_1Sq^2(c_3)=c_1^2c_3+c_1^2c_3=0,\]
	\begin{align*}
	&Sq^2\pi^*(y_{12})=c_1^2Sq^2(c_4)+Sq^2(c_1)c_2c_3+c_1Sq^2(c_2)c_3+c_1c_2Sq^2(c_3)\\
&=c_1^3c_4+c_1^2c_2c_3+c_1(c_1c_2+c_3)c_3+c_1^2c_2c_3=\pi^*(y_2y_{12}),
\end{align*}
	\begin{align*}
	&Sq^4\pi^*(y_8)=Sq^2(c_1)Sq^2(c_3)+c_1Sq^4(c_3)+[Sq^2(c_2)]^2\\
	&=c_1^3c_3+c_1(c_1c_4+c_2c_3)+(c_1c_2+c_3)^2=\pi^*(y_2^2y_8+y_{12}),
\end{align*}
\begin{align*}
	&Sq^4\pi^*(y_{12})=Sq^4(c_1^2)c_4+c_1^2Sq^4(c_4)+[Sq^2(c_3)]^2+Sq^4(c_1\cdot c_2c_3)\\
	&=c_1^4c_4+c_1^2c_2c_4+c_1^2c_3^2+[Sq^2(c_1)Sq^2(c_2c_3)+c_1Sq^4(c_2c_3)]\\
	&=c_1^4c_4+c_1^2c_2c_4+c_1^2c_3^2+[c_1^2c_3^2+c_1(c_1^2c_2c_3+c_1c_3^2+c_1c_2c_4)]\\
	&=c_1^4c_4+c_1^2c_3^2+c_1^3c_2c_3=\pi^*(y_2^2y_{12}),
\end{align*}
\begin{align*}
&Sq^8\pi^*(y_{12})=c_1^2Sq^8(c_4)+Sq^4(c_1^2)Sq^4(c_4)+[Sq^4(c_3)]^2+Sq^8(c_1c_2\cdot c_3)\\
	&=c_1^2c_4^2+c_1^4c_2c_4+(c_1c_4+c_2c_3)^2+Sq^8(c_1c_2\cdot c_3)=\pi^*(y_8y_{12}).
\end{align*}
In the last equation, we use
	\begin{align*} 
		&Sq^8(c_1c_2\cdot c_3)=Sq^6(c_1c_2)Sq^2(c_3)+Sq^2(c_1c_2)Sq^6(c_3)+Sq^4(c_1c_2)Sq^4(c_3)\\
		&=c_1^3c_2^2c_3+c_1c_3^3+(c_1^3c_2+c_1^2c_3+c_1c_2^2)(c_1c_4+c_2c_3).
	\end{align*}
\end{proof}

Next we consider $Sq^i$ on $\im\bar\Delta^*$. The following proposition contains all we need.
\begin{prop}\label{prop:steenrod delta}
	$Sq^2\bar\Delta^*(y_8)=w_2'^2w_3'^2$, $Sq^2\bar\Delta^*(y_9)=w_2'w_3'^3$, 
	 $Sq^4\bar\Delta^*(y_8)=w_3'^4+w_2'^3w_3'^2$,
	$Sq^4\bar\Delta^*(y_9)=w_2'^2w_3'^3$, $Sq^4\bar\Delta^*(y_{12})=w_2'^2w_3'^4$, 
\end{prop}
\begin{proof}
	By Proposition \ref{prop:even generator}, $\bar\Delta^*(y_8)=w_2'^4+w_2'w_3'^2$, $\bar\Delta^*(y_9)=w_2'^3w_3'$ and $\bar\Delta^*(y_{12})=w_2'^3w_3'^2$.	Using the Wu formula and the Cartan formula, it is straightforward to verify the formulas in the proposition. $Sq^2\bar\Delta^*(y_8)=w_2'^2w_3'^2$ is obvious.
	\begin{align*}
		&Sq^2\bar\Delta^*(y_9)=Sq^2(w_2'\cdot w_2'^2)w_3'+w_2^3Sq^2(w_3')\\
		&=(w_2'^4+w_2'w_3'^2)w_3'+w_2'^4w_3'=w_2'w_3'^3,
	\end{align*}
\[
Sq^4\bar\Delta^*(y_8)=[Sq^1(w_2')]^4+w_2'[Sq^4(w_3'^2)]=w_3'^4+w_2'^3w_3'^2,
\]
		\begin{align*}
		&Sq^4\bar\Delta^*(y_9)=Sq^4(w_2'\cdot w_2'^2)w_3'+Sq^2(w_2'\cdot w_2'^2)Sq^2(w_3')+Sq^1(w_2'^3)w_3'^2\\
		&=(w_2'^5+w_2'^2w_3'^2)w_3'+(w_2'^4+w_2'w_3'^2)w_2'w_3'+w_2'^2w_3'^3=w_2'^2w_3'^3,
	\end{align*}
		\begin{align*}
		&Sq^4\bar\Delta^*(y_{12})=Sq^4(w_2'\cdot w_2'^2)w_3'^2+w_2'^3Sq^4(w_3'^2)\\
		&=(w_2'^5+w_2'^2w_3'^2)w_3'^2+w_2'^5w_3'^2=w_2'^2w_3'^4.
	\end{align*}
\end{proof}

For $Sq^i$ on $\im\phi^*$, we only need a partial result.
\begin{prop}\label{prop:steenrod phi}
$Sq^2\phi^*(y_{12})=w_2^4w_3^2+w_2w_3^4+w_3^2w_4^2$, and \[Sq^8\phi^*(y_{12})\equiv\phi^*(y_8)\phi^*(y_{12})+w_3^4w_4^2\mod (w_2).\]
\end{prop}
\begin{proof}
Recall from Proposition \ref{prop:image in BSO(6)} that $\phi^*(y_{12})=w_6^2+w_2^2w_4^2+w_2w_5^2+w_2^3w_3^2$. One easily computes that $Sq^2(w_6^2)=0$, $Sq^2(w_2^2w_4^2)=w_2^2w_5^2+w_3^2w_4^2$, $Sq^2(w_2w_5^2)=w_2^2w_5^2$, $Sq^2(w_2^3w_3^2)=w_2^4w_3^2+w_2w_3^4$, hence the formula for $Sq^2\phi^*(y_{12})$ holds.

Now we compute the action of $Sq^8$ on each terms in the expression of $\phi^*(y_{12})$, modulo the ideal $(w_2)$. $Sq^8(w_2w_5^2)\equiv 0$ mod $(w_2)$ is obvious. 
\[Sq^8(w_6^2)=[Sq^4(w_6)]^2=w_4^2w_6^2\equiv \phi^*(y_8)\phi^*(y_{12})\mod (w_2).\]
\[\begin{split}
	Sq^8(w_2^3w_3^2)&=Sq^2(w_2^3)Sq^6(w_3^2)+Sq^4(w_2^3)Sq^4(w_3^2)+Sq^6(w_2^3)Sq^2(w_3^2)\\
	&\equiv 0\mod (w_2),
\end{split}\]
since $Sq^2(w_2^3)=w_2^4+w_2w_3^2$, $Sq^4(w_2^3)=w_2^5+w_2^2w_3^2$ and $Sq^6(w_2^3)=w_2^6$.
\[\begin{split}
Sq^8(w_2^2w_4^2)&=[Sq^4(w_2w_4)]^2=[w_2w_4^2+w_3Sq^3(w_4)+w_2^2Sq^2(w_4)]^2\\
&\equiv [w_3Sq^3(w_4)]^2\equiv w_3^4w_4^2\mod (w_2).
\end{split}\]
Then the formula for $Sq^8\phi^*(y_{12})$ follows.
\end{proof}
Now we are ready to prove Theorem \ref{thm:steenrod}. The action of $Sq^1$ on all generators $y_i$, and $Sq^2(y_3)$, $Sq^4(y_5)$ were given in Section \ref{sec:on the mod 2 cohomology}. $Sq^2(y_5)=0$ since $H^7(BPU(4);\Zz/2)=0$.

Since $H^{10}(BPU(4);\Zz/2)\cong\Zz/2\{y_2^5,y_2y_8,y_5^2\}$, Proposition \ref{prop:steenrod pi}, \ref{prop:steenrod delta} and \ref{prop:even generator}  give $Sq^2(y_8)=y_5^2$. Similarly, we have $Sq^4(y_8)=y_2^2y_8+y_{12}+y_3^4$, using the fact that  $H^{12}(BPU(4);\Zz/2)\cong\Zz/2\{y_2^6,y_2^2y_8,y_{12},y_3y_9,y_3^4\}$.  

Since $H^{11}(BPU(4);\Zz/2)\cong\Zz/2\{y_3^2y_5,y_3y_8\}$ and $\bar\Delta^*(y_3^2y_5)=w_2'w_3'^3$, $\bar\Delta^*(y_3y_8)=w_3'(w_2'^4+w_2'w_3'^2)$ by Proposition \ref{prop:even generator}, it follows from  Proposition \ref{prop:steenrod delta} that $Sq^2(y_9)=y_3^2y_5$. Similarly, the formula $Sq^4(y_9)=y_3y_5^2$ follows from Proposition \ref{prop:steenrod delta}, and the fact that  $H^{13}(BPU(4);\Zz/2)\cong\Zz/2\{y_3y_5^2,y_5y_8\}$, using $\bar\Delta^*(y_3y_5^2)=w_2'^2w_3'^3$ and $\bar\Delta^*(y_5y_8)=w_2'w_3'(w_2'^4+w_2'w_3'^2)$. 

Since $H^{17}(BPU(4);\Zz/2)\cong\Zz/2\{y_3^4y_5, y_3^3y_8, y_3y_5y_9, y_5y_{12}, y_8y_9\}$, we may assume that $Sq^8(y_9)=ay_3^4y_5+by_3^3y_8+cy_3y_5y_9+dy_5y_{12}+ey_8y_9$. Then \[Sq^1Sq^8(y_9)=ay_3^6+by_3^6+c(y_3^3y_9+y_3y_5^3)+d(y_3^2y_{12}+y_3y_5^3)+e(y_3^3y_9+y_5^2y_8).\]
Hence, the relation $Sq^1Sq^8(y_9)=Sq^9(y_9)=y_9^2=y_3^2y_{12}+y_5^2y_8$ gives $a=b$ and $c=d=e=1$. On the other hand, if $a=b=1$, then $w_2'^4w_3'^3$ would not appear in the expression of $Sq^8\bar\Delta^*(y_9)=Sq^8(w_2'^3w_3')$, but this is not true by an easy computation, so $Sq^8(y_9)=y_3y_5y_9+y_5y_{12}+y_8y_9$.

Notice that $H^{14}(BPU(4);\Zz/2)$ is spanned by $y_2^7$, $y_2^3y_8$, $y_2y_{12}$, $y_3^3y_5$, $y_3^2y_8$, $y_5y_9$. Then in the expression $Sq^2(y_{12})=ay_2^7+by_2^3y_8+cy_2y_{12}+dy_3^3y_5+ey_3^2y_8+fy_5y_9$, $a=b=0$, $c=1$ by Proposition \ref{prop:steenrod pi}, and $d=f=0$, $e=1$ by Proposition \ref{prop:steenrod phi} and the fact that $\phi^*(y_3^3y_5)\equiv w_3^3w_5$, $\phi^*(y_3^2y_8)\equiv w_3^2w_4^2$, $\phi^*(y_5y_9)\equiv w_5(w_3w_6+w_4w_5)$ mod $(w_2)$ from Proposition \ref{prop:image in BSO(6)}, which proves $Sq^2(y_{12})=y_2y_{12}+y_3^2y_8$.

 Since $H^{16}(BPU(4);\Zz/2)\cong\Zz/2\{y_2^8,y_2^4y_8,y_2^2y_{12},y_8^2,y_3^2y_5^2,y_3y_5y_8\}$, Proposition \ref{prop:steenrod pi}, \ref{prop:steenrod delta} and \ref{prop:even generator} give $Sq^4(y_{12})=y_2^2y_{12}+y_3^2y_5^2$. 
 
 An easy calculation shows that $H^{20}(BPU(4);\Zz/2)$ is spanned by $y_3^5y_5$, $y_3^4y_8$, $y_3^2y_5y_9$, $y_3y_5y_{12}$, $y_3y_8y_9$, $y_5^4$ and  monomials in $y_2$, $y_8$, $y_{12}$ of degree $20$. Proposition \ref{prop:steenrod pi} says that among monomials in $y_2$, $y_8$, $y_{12}$ of degree $20$, only $y_8y_{12}$ appears in the expression of $Sq^8(y_{12})$. Furthermore, since $\phi^*(y_3^5y_5)\equiv w_3^5w_5$, $\phi^*(y_3^4y_8)\equiv w_3^4w_4^2$, $\phi^*(y_3^2y_5y_9)\equiv w_3^2w_5(w_3w_6+w_4w_5)$, $\phi^*(y_3y_5y_{12})\equiv w_3w_5w_6^2$, $\phi^*(y_3y_8y_9)\equiv w_3w_4^2(w_3w_6+w_4w_5)$, $\phi^*(y_5^4)\equiv w_5^4$ mod $(w_2)$ by Proposition \ref{prop:image in BSO(6)}, we obtain from Proposition \ref{prop:steenrod phi} that $Sq^8(y_{12})=y_3^4y_8+y_8y_{12}$. 

The proof of Theorem \ref{thm:steenrod} is completed.
\section{Proof of Theorem \ref{thm:BPU4}}
Let $\bar\alpha_i\in H^{2i}(BPU(4);\Zz)$, $i=2,3,4,6$, be the elements constructed in Section \ref{sec:^UE}, and for simplicity let $x_1$, $y_{2,I}$ denote $\chi^*(x_1)$, $\chi^*(y_{2,I})\in H^*(BPU(4);\Zz)$ for $\chi:BPU(4)\to K(\Zz,3)$.

\begin{prop}\label{prop:generated}
	The cohomology ring $H^*(BPU(4);\Zz)$ is generated by $\bar\alpha_i$, $i=2,3,4,6$, and $x_1$, $y_{2,1}$, $y_{2,(1,0)}$.
\end{prop}
Before proving Proposition \ref{prop:generated}, we need a result on the image of $H^*(BPU(4);\Zz)$ under the mod $2$ reduction.
\begin{lem}\label{lem:image rho}
The image of $\rho:H^*(BPU(4);\Zz)\to H^*(BPU(4);\Zz/2)$ is the subalgebra generated by $y_2^2$, $y_2^3$, $y_3$, $y_5^2$, $y_8+y_3y_5$, $y_{12}+y_3y_9$, $y_3^2y_9+y_5^3$.
\end{lem}

\begin{proof}
	Let $S\subset H^*(BPU(4);\Zz/2) $ be the subring generated by the elements in the Lemma. First we prove that $S\subset\im\rho$. By Proposition \ref{prop:Z_2 alpha_i},  $\rho(\bar\alpha_2)=y_2^2$, $\rho(\bar\alpha_3)=y_2^3$, $\rho(\bar\alpha_4)=y_8+y_3y_5$, $\rho(\bar\alpha_6)=y_{12}+y_3y_9$. Recall from Section \ref{sec:K3} that in $H^*(K(\Zz,3);\Zz/2)$ we have $y_{2,0}=x_1^2$, $y_{2,1}=x_{2,0}^2$ and $y_{2,(1,0)}=y_{2,0}x_{2,1}+x_{2,0}y_{2,1}=x_1^2x_{2,1}+x_{2,0}^3$. Hence from Proposition \ref{prop:odd generator} and the second part of Proposition \ref{prop:w_i} we have $\rho(x_1)=y_3$, $\rho(y_{2,1})=y_5^2$, $\rho(y_{2,(1,0)})=y_3^5+y_3^2y_9+y_5^3$. It follows immediately that $S\subset\im\rho$.
	
	For the reverse inclusion $\im\rho\subset S$, we use the fact that $\im\rho=\ker\beta\subset \ker Sq^1$, where $\beta$ is the Bockstein homomorphism in the long exact sequence 
	\[\cdots\to H^i(-;\Zz)\xrightarrow{\cdot 2} H^i(-;\Zz)\xrightarrow{\rho}H^i(-;\Zz/2)\xrightarrow{\beta}H^{i+1}(-;\Zz)\to\cdots\]
	induce by the sequence $0\to\Zz\xrightarrow{\cdot2}\Zz\to\Zz/2\to 0$, noting that $Sq^1=\rho\beta$. We claim that the kernel of $Sq^1$ on $H^*(BPU(4);\Zz/2)$ is the subalgebra generated by $y_2$, $y_3$, $y_5^2$, $y_8+y_3y_5$, $y_{12}+y_3y_9$, $y_3^2y_9+y_5^3$, whose proof is defered to the last Section for reader's convenience. Then from the ring structure of $H^*(BPU(4);\Zz/2)$ it is easy to see that there is a $\Zz/2$-module isomorphism 
	\[\ker Sq^1\cong S\oplus y_2\cdot\Zz/2[y_8,y_{12}].\]
	Furthermore, we know that $H^2(BPU(n);\Zz)=0$ and $H^3(BPU(n);\Zz)=\Zz/n\{x_1\}$ for any $n$. Hence the universal coefficient theorem shows that $H^2(BPU(4);\Zz/4)\cong\Zz/4$.
The sequence $0\to\Zz\xrightarrow{\cdot4}\Zz\to\Zz/4\to 0$ also induces a long exact sequence of cohomology, and there is a commutative diagram
	\[\xymatrix{\ar[r]&H^i(-;\Zz)\ar[r]^{\cdot4}\ar^{\cdot2}[d]&H^i(-;\Zz)\ar[r]^{\rho'}\ar@{=}[d]&H^i(-;\Zz/4)\ar[r]^{\beta'}\ar[d]^{\theta=\mathrm{mod}\ 2}&H^{i+1}(-;\Zz)\ar[r]\ar^{\cdot2}[d]&\\
	\ar[r]&H^i(-;\Zz)\ar[r]^{\cdot2}&H^i(-;\Zz)\ar[r]^{\rho}&H^i(-;\Zz/2)\ar[r]^{\beta}&H^{i+1}(-;\Zz)\ar[r]&}
	\]
	 Suppose that $y_2'$ generates $H^2(BPU(4);\Zz/4)$. Then $\beta'(y_2')=x_1$, and $\theta(y_2')=y_2$.
	Hence, the above diagram shows that $\beta(y_2)=2x_1$, and then 
	\begin{gather*}
	\beta(y_2y_8^iy_{12}^j)=\beta(y_2(y_8+y_3y_5)^i(y_{12}+y_3y_9)^j)=\beta(y_2\rho(\bar\alpha_4^i\bar\alpha_6^j))=2x_1\bar\alpha_4^i\bar\alpha_6^j.
	\end{gather*}
This implies that for any $f\in\Zz/2[y_8,y_{12}]$, $\beta(y_2f)\neq 0$ by the second part of Proposition \ref{prop:K_4}. Since $\beta\rho=0$, it follows that
$(y_2\cdot\Zz/2[y_8,y_{12}])\cap\im\rho=0$, and then $\im\rho\subset S$.
\end{proof}

\begin{proof}[Proof of Proposition \ref{prop:generated}]
Let $R$ be the subring generated by the elements in the proposition, and let $N=H^*(BPU(4);\Zz)/R$ be the quotient $\Zz$-module. Then lemma \ref{lem:image rho} and the first paragraph of its proof show that $\rho(R)=\rho(H^*(BPU(4);\Zz))$, i.e. $\rho(N)=0$.
We need to show that $N$ itself is zero.
	
Let $\mathbf{t}$ be the ideal consisting of torsion elements of $H^*(BPU(4);\Zz)$, and let $F=H^*(BPU(4);\Zz)/\mathbf{t}$. Since $\pi^*(\bar\alpha_i)=\alpha_i$, using Proposition \ref{prop:K_4} and the fact that torsion free elements in $^UE$ concentrated in $^UE_*^{0,*}$, we see that
\[\pi^*(R)\cong K_4\cong {^U}E_\infty^{0,*}\cong F.\]
Hence we get a commutative diagram of short exact sequences
\[\xymatrix{0\ar[r]&\mathbf{t}\cap R\ar[r]\ar[d]&R\ar[r]\ar[d]&F\ar[r]\ar@{=}[d]&0\\
0\ar[r]&\mathbf{t}\ar[r]&H^*(BPU(4);\Zz)\ar[r]&F\ar[r]&0}
\]
This implies that $N\cong \mathbf{t}/(\mathbf{t}\cap R)$ as $\Zz$-modules.
Recall that the torsion elements of $H^*(BPU(4);\Zz)$ are all $2$-primary. So if $N\neq 0$, then $N\otimes\Zz/2\neq 0$.
However, since $0=\rho(N)\cong N\otimes\Zz/2$, we must have $N=0$. Hence $R=H^*(BPU(4);\Zz)$.
\end{proof}

Let $A=K_4\otimes\Zz[x_3,x_{10},x_{12}]/I$ be the ring defined in Theorem \ref{thm:BPU4}.
\begin{prop}\label{prop:surjective}
	There is a surjective homomorphism of graded rings: 
	\[A\to H^*(BPU(4);\Zz),\ \ x_3\mapsto x_1,\ x_{10}\mapsto y_{2,1},\ x_{15}\mapsto y_{2,(1,0)},\ \alpha_i\mapsto\bar\alpha_i.\]
\end{prop}
\begin{proof}
By Proposition \ref{prop:generated}, we only need to show that this map is well-defiend, in other words, the element $\alpha=64\alpha_6-97\alpha_2^3-27\alpha_3^2+144\alpha_2\alpha_4$ and all generators of the ideal $I$ map to zero.

We already know that in $H^*(BPU(4);\Zz)$, $x_1$ has order $4$ by \cite[Theorem 1.1]{Gu21}, and $x_1^2$, $y_{2,1}$, $y_{2,(1,0)}$ are all of order $2$ by Theorem \ref{thm:cohomology of k(3)} and the fact that  $\rho(x_1^2)$, $\rho(y_{2,1})$, $\rho(y_{2,(1,0)})\neq 0$ in $H^*(BPU(4);\Zz/2)$. 
From this one can see that the order of a torsion element of even degree in $H^*(BPU(4);\Zz)$ is at most $2$.

Let $\bar\alpha=64\bar\alpha_{6}-97\bar\alpha_2^3-27\bar\alpha_3^2+144\bar\alpha_2\bar\alpha_4$. Then $\pi^*(\bar\alpha)=\alpha=0$ for $\pi:BU(4)\to BPU(4)$, since $\im\pi^*\cong K_4$ by Proposition \ref{prop:K_4}. Hence $\bar\alpha$ is a torsion element. From Proposition \ref{prop:Z_2 alpha_i} we know that $\rho(\bar\alpha)=0$, and so $\bar\alpha\in 2H^*(BPU(4);\Zz)$. Then we get $\bar\alpha=0$ since $\bar\alpha$ is of even degree.

Since $y_2y_5=0$ and $\rho(y_{2,1})=\chi^*(x_{2,0}^2)=y_5^2$, we have $\rho(\bar\alpha_2y_{2,1})=\rho(\bar\alpha_3y_{2,1})=0$. Hence $\bar\alpha_2y_{2,1}$, $\bar\alpha_3y_{2,1}\in 2H^*(BPU(4);\Zz)$, and then $\bar\alpha_2y_{2,1}=\bar\alpha_3y_{2,1}=0$, since $\bar\alpha_2y_{2,1}$, $\bar\alpha_3y_{2,1}$ are of even degree.

It is easy to see that $H^7(BPU(4);\Zz)=0$ since $H^7(BPU(4);\Zz/2)=0$, so $\bar\alpha_2x_1=0$.  Since $\rho(\bar\alpha_3x_1)=y_2^3y_3=0$, $\bar\alpha_3x_1\in 2H^9(BPU(4);\Zz)$. Then the fact that $H^9(BPU(4);\Zz)\cong \Zz/2\{x_1^3\}$ gives $\bar\alpha_3x_1=0$.

Since $\rho(y_{2,(1,0)})=\chi^*(x_1^2x_{2,1}+x_{2,0}^3)=y_3^5+y_3^2y_9+y_5^3$, we have $\rho(\bar\alpha_2y_{2,(1,0)})=0$, then $\bar\alpha_2y_{2,(1,0)}\in2H^{19}(BPU(4);\Zz)$. Checking each  element of degree $19$, one can show that $\bar\alpha_2y_{2,(1,0)}=0$ or $2\bar\alpha_4^2x_1$. But the later is impossible, since if this is the case, it would imply that $\alpha_4^2x_1$ in $^UE_\infty^{3,16}$ has order $2$ because of the Serre filtration construction, which contradicts the second part of Proposition \ref{prop:K_4}.
Hence we get the relation $\bar\alpha_2y_{2,(1,0)}=0$. The proof of $\bar\alpha_3y_{2,(1,0)}=0$ is easier, using the fact that all elements of degree $21$ have order $2$.

Let $y=x_1^{6}\bar\alpha_6+x_1^4y_{2,1}\bar\alpha_4+x_1^{5}y_{2,(1,0)}+y_{2,1}^3+y_{2,(1,0)}^2$. Then using the computations in the first paragraph of the proof Lemma \ref{lem:image rho} and the relation $y_9^2=y_3^2y_{12}+y_5^2y_8$ in $H^*(BPU(4);\Zz/2)$, one readily checks that $\rho(y)=0$. Since $y$ has even degree, $y=0$ by the same reasoning.

The above calculations verify that $\alpha$ and all generators of $I$ map to zero, finishing the proof.
\end{proof} 


The following Proposition completes the proof of Theorem \ref{thm:BPU4}.
\begin{prop}
The homomorphism in Proposition \ref{prop:surjective} is injective.
\end{prop}
\begin{proof}
The proposition is equivalent to saying that if a polynomial $f$ in the variables $\alpha_i$, $i=2,3,4,6$, $x_3$, $x_{10}$, $x_{15}$, maps to zero in $H^*(BPU(4);\Zz)$, then $f\in(\alpha)+I$, where $\alpha=64\alpha_6-97\alpha_2^3-27\alpha_3^2+144\alpha_2\alpha_4$. 

Write $f=f(\bar\alpha_2,\bar\alpha_3,\bar\alpha_4,\bar\alpha_6,x_1,y_{2,1},y_{2,(1,0)})$.
Since $f$ is zero in $H^*(BPU(4);\Zz)$, $\pi^*(f)=0$, and then $f$ can be written as
$f=f_0+f_1$, where $f_0$ belongs to the ideal of $\Zz[\bar\alpha_2,\bar\alpha_3,\bar\alpha_4,\bar\alpha_6]$ generated by $\bar\alpha=64\bar\alpha_{6}-97\bar\alpha_2^3-27\bar\alpha_3^2+144\bar\alpha_2\bar\alpha_4$, and $f_1\in (x_1,y_{2,1},y_{2,(1,0)})$. Hence it suffices to show that $f_1\in I$. Here we identify $I$ with the corresponding ideal of $\Zz[\bar\alpha_2,\bar\alpha_3,\bar\alpha_4,\bar\alpha_6,x_1,y_{2,1},y_{2,(1,0)}]$. 

Let $\phi:BSO(6)\to BPSO(6)=BPU(4)$ be the map in Section \ref{sec:on the mod 2 cohomology}, and let 
\begin{gather*}
g_1=\phi^*\rho(x_1)=\phi^*(y_3),\ \ g_2=\phi^*\rho(y_{2,1})=\phi^*(y_5^2),\\ g_3=\phi^*\rho(\bar\alpha_4)=\phi^*(y_8+y_3y_5),\ \ g_4=\phi^*\rho(\bar\alpha_6)=\phi^*(y_{12}+y_3y_9),\\ 
g_5=\phi^*\rho(y_{2,(1,0)})=\phi^*(y_3^5+y_3^2y_9+y_5^3).
\end{gather*}
Then $\phi^*\rho(f_1)$ belongs to the ideal $I_1$ of relations between the elements $g_1,\dots,g_5$ in $H^*(BSO(6);\Zz/2)\cong \Zz/2[w_2,\dots,w_6]$.

Let us compute the ideal $I_1\subset \Zz/2[g_1,\dots,g_5]$. By Proposition \ref{prop:image in BSO(6)}, modulo $(w_2)$, 
\begin{gather*}
g_1\equiv w_3,\ \ g_2\equiv w_5^2,\ \ g_3\equiv w_4^2+w_3w_5,\\ 
	g_4\equiv w_6^2+w_3^2w_6+w_3w_4w_5,\ \ g_5\equiv w_3^5+w_3^3w_6+w_3^2w_4w_5+w_5^3.
\end{gather*}
It is easy to see that the four elements $g_1,\dots,g_4$ are algebraically independent. Hence the Krull dimension of the subalgbra of $\Zz/2[w_2,\dots,w_6]$ generated by $g_i$, $i=1,\dots,5$, is at least $4$, which means that $I_1$ is a principal ideal of $\Zz/2[g_1,\dots,g_5]$. To see this, note that $I_1$ is the kernel of the map $\Zz/2[g_1,\dots,g_5]\to \Zz/2[w_2,\dots,w_6]$ with target a UFD, so $I_1$ is a prime ideal of $\Zz/2[g_1,\dots,g_5]$. Since the Krull dimension of $\Zz/2[g_1,\dots,g_5]/I_1$ is at least $4$, the height of $I_1$ is at most $1$, and then $I_1$ is principal since $\Zz/2[g_1,\dots,g_5]$ is a UFD.

 Suppose that $h=h(g_1,\dots,g_5)$ generates $I_1$. 
Since the powers $w_6^i$ only appear in the expressions of $g_4$ and $g_5$, it follows that if $h\neq 0$, then the degrees of $g_4$ and $g_5$ in $h$ are both nonzero. Furthermore, since the degree of $w_6$ in $g_4$ (resp. in $g_5$) is $2$ (resp. $1$), the degree of $g_5$ in $h$ is at least $2$ when $h\neq 0$. We have seen in the proof of Proposition \ref{prop:surjective} that $\phi^*\rho(y)=g_1^6g_4+g_1^4g_2g_3+g_1^5g_5+g_2^3+g_5^2=0$ in $\Zz/2[w_2,\dots,w_6]$, where $y=x_1^{6}\bar\alpha_6+x_1^4y_{2,1}\bar\alpha_4+x_1^{5}y_{2,(1,0)}+y_{2,1}^3+y_{2,(1,0)}^2$, hence we must have $h=g_1^6g_4+g_1^4g_2g_3+g_1^5g_5+g_2^3+g_5^2$.

Since $f_1\in (x_1,y_{2,1},y_{2,(1,0)})$ and $\phi^*\rho(\bar\alpha_2)=\phi^*\rho(\bar\alpha_3)=0$, the above analysis implies that 
\[f_1\in (y)+(\bar\alpha_2,\bar\alpha_3)\cdot(x_1,y_{2,1},y_{2,(1,0)})+2(x_1,y_{2,1},y_{2,(1,0)}).\]
We know that $f$ and $f_0$ are zero in $H^*(BPU(4);\Zz)$, so $f_1=0$ in $H^*(BPU(4);\Zz)$. Since $y$, $2y_{2,1}$, $2y_{2,(1,0)}$ and $(\bar\alpha_2,\bar\alpha_3)\cdot(x_1,y_{2,1},y_{2,(1,0)})$ are all zero in $H^*(BPU(4);\Zz)$ and the order of $x_1\xi$ is $4$ for any $\xi\in\Zz[\bar\alpha_4,\bar\alpha_6]$ by the second part of Proposition \ref{prop:K_4}, it follows that $f_1$ is actually in the ideal \[(y)+(\bar\alpha_2,\bar\alpha_3)\cdot(x_1,y_{2,1},y_{2,(1,0)})+4(x_1),\]
which is contained in $I$. The proof is completed.
\end{proof}

\section{Kernel of $Sq^1$ on $H^*(BPU(4);\Zz/2)$}
In this section we prove the  claim in the proof of Lemma \ref{lem:image rho} that the kernel of the action $Sq^1$ on $H^*(BPU(4);\Zz/2)$ is the subalgebra generated by $y_2$, $y_3$, $y_5^2$, $y_8+y_3y_5$, $y_{12}+y_3y_9$, $y_3^2y_9+y_5^3$. 
To make the computation easier, we shall establish a $\Zz/2$-algebra isomorphic to $H^*(BPU(4);\Zz/2)$, and define a operation on it corresponding to $Sq^1$. 

 Let $W$ be the quotient of the polynomial algebra $\Zz/2[x_2,x_3,x_5,x_8,x_9,x_{12}]$ by the ideal generated by $x_2x_3$, $x_2x_5$, $x_2x_9$ and $x_9^2+x_3^2x_{12}+x_5^2x_8+x_3^3x_9+x_3x_5^3$. Then it is easy to see that there is an algebra isomporphism $W\cong H^*(BPU(4);\Zz/2)$ given by $x_i\mapsto y_i$ for $i=2,3,5,9$, $x_8\mapsto y_8+y_3y_5$, $x_{12}\mapsto y_{12}+y_3y_9$. Define an operation $\DD$ on $W$ by $\DD(x_i)=0$ for $i=2,3,8,12$, $\DD(x_5)=x_3^2$, $\DD(x_{9})=x_5^2$, and the  Leibniz rule $\DD(ab)=\DD(a)b+a\DD(b)$. Then using Theorem \ref{thm:steenrod}, one easily verifies that there is a commutative diagram 
\[\xymatrix{
	W\ar[r]^{\DD}\ar[d]_\cong&W\ar[d]_\cong\\
	H^*(BPU(4);\Zz/2)\ar[r]^{Sq^1}&H^*(BPU(4);\Zz/2)}\]
Thus, $W$ can be viewed as a DGA with differential $\DD$. We compute its cohomology as follows
\begin{prop}\label{prop:bockstein}
	$H^*(W,\DD)\cong\Zz/2[x_2,x_8,x_{12}]\otimes E_{\Zz/2}[x_3]$.
\end{prop}
Our desired result on $\ker Sq^1$ is a direct consequence of the following corollary of Proposition \ref{prop:bockstein}. 
\begin{cor}
$\ker\DD$ is the subalgebra of $W$ generated by $x_2$, $x_8$, $x_{12}$, $x_3$, $x_5^2$, $x_3^2x_9+x_5^3$.
\end{cor}
\begin{proof}
Since $H^*(W,\DD)=\ker\DD/\im\DD$, it follows from Proposition \ref{prop:bockstein} that as $\Zz/2$-modules
\[\ker\DD\cong \im\DD\oplus\Zz/2[x_2,x_8,x_{12}]\oplus\Zz/2[x_2,x_8,x_{12}]\cdot x_3.\]
 Now we compute $\im\DD$.
For a monomial $\mm=\nn\cdot x_3^ix_5^jx_9^k\in W$, $\nn\in\Zz/2[x_2,x_8,x_{12}]$, by the rule of $\DD$ we have
\[\DD(\mm)=\nn\cdot(\w j\cdot x_3^{i+2}x_5^{j-1}x_9^{k}+\w k\cdot x_3^{i}x_5^{j+2}x_9^{k-1}).\]
Here $\w *$ means the mod $2$ number. Hence $\im\DD$ is a $\Zz/2$-vector space spanned by elements of the forms:
\[\nn\cdot x_3^rx_5^{2s}x_9^{2t},\ \ \nn\cdot x_3^ix_5^{2q}x_9^{2t},\ \  \nn\cdot x_3^ix_5^{2s}x_9^{2t}(x_3^2x_9+x_5^3),\]
where $i,s,t\geq 0$,  $q\geq 1$, $r\geq2$. Note that $x_9^{2t}=[x_3^2x_{12}+x_5^2x_8+x_3(x_3^2x_9+x_5^3)]^t$ by definition. Then an easy calculation gives the result.
\end{proof}

\begin{proof}[Proof of Proposition \ref{prop:bockstein}]
	Define a $\Zz/2$-linear map $\lambda:W\to W$ by $\lambda(\mm)=\mm$ if $\mm=x_2^ax_8^bx_{12}^c$ or $x_2^ax_8^bx_{12}^cx_3$, $a,b,c\geq 0$, and $\lambda(\mm)=0$ for all other monomials. It is easy to show that $\lambda$ is a chain map associated to the differential $\DD$. We claim that $\lambda$ is chain homotopic to the identity map. Since $\DD$ restricted to $\Zz/2[x_2,x_8,x_{12},x_3]$ is zero and any element in  $\Zz/2[x_2,x_8,x_{12}]\cdot\Zz/2\{1,x_3\}$ is clearly not in $\im\DD$, the proposition follows immediately from this claim.
	
	Now we construct the chain homotopy $P$ as follows. For a monomial $\mm=\nn\cdot x_3^ix_5^jx_9^k\in W$, $\nn\in\Zz/2[x_2,x_8,x_{12}]$, let
	\[P(\mm)=\begin{cases}
	\nn\cdot x_3^{i-2}x_5^{j+1}x_9^k&\text{if }i\geq 2,\\
	\nn\cdot x_3^ix_5^{j-2}x_9^{k+1}&\text{if }i\leq1,\ j,k \text{ even},\ j\neq 0,\\
	\nn\cdot x_3^ix_9^{k-2}(x_5x_{12}+x_8x_9+x_3x_5x_9)&\text{if }i\leq1,\ j=0,\ k\geq 2\ \text{even},\\
	0&\text{otherwise}.
	\end{cases}\]
	It is straightforward to show that $P\DD+\DD P=\lambda-\mathrm{id}$ on all monomials in $W$ case by case.
\end{proof}
	
\bibliography{M-A}
\bibliographystyle{amsplain}
\end{document}